\documentclass[12pt]{amsart}
 \usepackage{amssymb,url,rotating,color,enumerate}
 \usepackage[left=2cm,right=2cm,top=2cm,bottom=2cm]{geometry}

 \newtheorem{theorem}{Theorem}

 \newtheorem{conjecture}[theorem]{Conjecture}

 \newtheorem{definition}[theorem]{Definition}
 \newtheorem{example}[theorem]{Example}

 \newtheorem{lemma}[theorem]{Lemma}

 \newtheorem{proposition}[theorem]{Proposition}
 \newtheorem{remark}[theorem]{Remark}

 \newcommand{\C}{\mathbb{C}}

 \newcommand{\R}{\mathbb{R}}

 \newcommand{\re}{\mathrm{Re}}
 \newcommand{\im}{\mathrm{Im}}

 \newcommand{\E}{\mathcal{E}^{\mathrm{par}}}
 \newcommand{\s}{\mathcal{S}}
 \newcommand{\shf}{\mathcal{S}^{\sharp\flat}}
 \newcommand{\shfs}{\mathcal{S}^{\sharp \flat }(\sigma_0, \sigma_1)}
 \newcommand{\abs}[1]{\left\vert#1\right\vert}

\begin{document}
 	
 	\title[On generalized Li criterion for a certain class of $L-$functions]
 	{On generalized Li criterion for a certain class of $L-$functions}
 	\author[A.-M. Ernvall-Hyt\"onen]{Anne-Maria Ernvall-Hyt\"onen}
 	\address{Department of Mathematics and Statistics, University of Helsinki, P.O. Box 68, 00014 Helsinki, Finland, e-mail: ernvall@mappi.helsinki.fi}
 	\author[A. Od\v{z}ak]{Almasa Od\v{z}ak}
 	\address{Department of Mathematics, University of Sarajevo, Zmaja od Bosne 35, 71\,000 Sarajevo,
 		Bosnia and Herzegovina, e-mail: almasa@pmf.unsa.ba}
 	\author[L.~Smajlovi\'c]{Lejla Smajlovi\'c}
 	\address{Department of Mathematics, University of Sarajevo, Zmaja od Bosne 35, 71\,000 Sarajevo,
 		Bosnia and Herzegovina, e-mail: lejlas@pmf.unsa.ba}
 	\author[M. Su\v si\'c ]{Medina Su\v si\'c}
 	\address{Department of Mathematics, University of Sarajevo, Zmaja od Bosne 35, 71\,000 Sarajevo,
 		Bosnia and Herzegovina, e-mail: medina.susic@pmf.unsa.ba}

\thanks{The ideas presented in the paper were discussed during the third WINE conference, held in October 2013 and funded by CIRM, Microsoft Research, Number theory foundation, NSF and Clay Mathematics Institute. Their support is gratefully acknowledged.
\\The research of Anne-Maria Ernvall-Hyt\"onen was supported by the Academy of Finland  grant no. 138337.}
\maketitle
\begin{abstract}

We define generalized Li coefficients, called $\tau-$Li coefficients for a very broad class $\shfs$ of $L-$functions that contains the Selberg class, the class of all automorphic $L-$functions and the Rankin-Selberg $L-$functions, as well as products of suitable shifts of those functions. We prove the generalized Li criterion for zero-free regions of functions belonging to the class $\shfs$, derive an arithmetic formula for the computation of $\tau-$Li coefficients and conduct numerical investigation of $\tau-$Li coefficients for a certain product of shifts of the Riemann zeta function.

Keywords:
$L$-functions,  Generalized Li coefficients, Generalized Li criterion

MSC: 11M26, 11M36, 11M41
\end{abstract}

\section{Introduction}


The Li criterion for the Riemann hypothesis, proved in \cite{XJLi} is a simple positivity criterion stating that the Riemann hypothesis is equivalent to non-negativity of a certain sequence of real numbers, called the Li coefficients. The Li criterion is generalized to many classes of functions. For Dirichlet and Hecke $L-$functions it is proved in \cite{XJLi04}, for automorphic $L-$functions the Li criterion is deduced in \cite{Lagarias}, for the Rankin-Selberg $L-$functions it is proved in \cite{OdSm09}. In \cite{Sm10}, a class $\s^{\sharp\flat}$ that contains both the Selberg class $\s$ and the class of automorphic $L-$functions is introduced and the Li criterion for this class is obtained.

The Li coefficients for various classes of zeta and $L-$functions may be generalized in different ways. A. Droll \cite{Dr12}, following \cite{Freitas}, defined for $\tau \in[1,2)$ and positive integers $n$ the generalized $\tau-$Li coefficients $\lambda_F(n,\tau)$, for $F\in\s^{\sharp\flat}$ as
\begin{equation} \label{tau Li coeff sum}
\lambda_F(n,\tau)=\left.\sum_{\rho\in Z(F)}\right. ^{\ast} \left( 1-\left( \frac{\rho}{\rho-\tau}\right)^n \right)
\end{equation}
where the sum is taken over the set $Z(F)$ of all non-trivial zeros of $F$ and $\ast$ denotes that the sum is taken in the sense of the limit $\lim\limits_{T \to \infty} \sum\limits_{|\im \rho| \leq T}$.

In \cite{Dr12} it is proved that non-negativity of $\re (\lambda_F(n,\tau))$ for all positive integers $n$ is equivalent to the statement that all non-trivial zeros of $F\in\shf$ are in the strip $1-\tau/2\leq \re s\leq \tau/2$. We will refer to this criterion as $\tau-$Li criterion. It is a generalization of the Li criterion in the sense that non-negativity of $\re (\lambda_F(n,1))$ is equivalent to the statement that all non-trivial zeros of $F$ are on the critical line $\re s=1/2$, i.e. non-negativity of $\re (\lambda_F(n,1))$ is equivalent to the generalized Riemann hypothesis for $F$. However, for $\tau >1$ the $\tau-$Li criterion is a weaker statement, as it produces only zero-free regions. The $\tau-$Li criterion for the Rankin-Selberg $L-$functions and asymptotic behavior of coefficients $\lambda_F(n,\tau)$ was deduced in \cite{WINEpaper}.

The main purpose of investigation conducted in this paper is to define $\tau-$Li coefficients for a very broad class of $L-$functions, we denote by $\s^{\sharp\flat}(\sigma_0, \sigma_1)$, prove the $\tau-$Li criterion for this class, derive an arithmetic formula for the computation of $\tau-$Li coefficients and investigate analytically and numerically the properties of $\tau-$Li coefficients.

The class $\s^{\sharp \flat}(\sigma_0, \sigma_1)$ consists of all Dirichlet series $F$ converging in some half-plane $\re s>\sigma_0\geq 1$, such that the meromorphic continuation of $F$ to $\C$ is a meromorphic function of a finite order with at most finitely many poles, satisfying a functional equation relating values $F(s)$ with $\overline{F(\sigma_1-\overline{s})}$ up to multiplicative gamma factors and such that the logarithmic derivative $F'/F$ has a Dirichlet series representation converging in the half plane $\re s>\sigma_0\geq1$. The assumptions posed on $F\in \s^{\sharp \flat}(\sigma_0, \sigma_1)$ imply that all non-trivial zeros of $F$ lie in the strip $\sigma_1-\sigma_0 \leq \re s \leq \sigma_0$. The class $\shfs$ contains $\shf$ as its subclass. Furthermore, the class $\shfs$ contains products of suitable shifts of $L-$functions from $\shf$, as well as products of shifts of certain $L-$functions possessing an Euler product representation that are not in $\shf$ (such as the Rankin-Selberg $L-$functions).

Another reason why we introduce the class $\shfs$ lies in the fact that, as proved in Theorem \ref{thm: tau li implies Li} below, for $F\in\shf$ and an arbitrary positive real number $a>0$ the $\tau-$Li criterion for the function $F(s-a)F(s+a)$ (that belongs to the class $\shf(a+1,1)$ and does not belong to $\shf$) with $\tau=2a+1$ is equivalent to the generalized Riemann hypothesis for $F$. On the other hand, as proved in Theorem \ref{thm: arithmetic f-la 1} the $\tau-$Li coefficients for $\tau=2a+1$ may be expressed in terms of a sum of derivatives of digamma functions, certain finite sums and a rapidly converging Dirichlet series (for large values of $a$), an expression that is suitable both for analytic considerations and numerical computations since the sum of derivatives of digamma functions may be easily estimated up to an error term $n^{-k}$, for an arbitrary positive integer $k$ (see e.g. computations in \cite{WINEpaper}, Section 6).

Let us note that for $F\in\shf$ the $\tau-$Li criterion for $\tau=1$ is equivalent to the generalized Riemann hypothesis, however, in this case an arithmetic expression for the Li coefficients is far from being rapidly converging, see Theorem 6.1. and Theorem 6.2. from \cite{Sm10}.

The paper is organized as follows. In Section 2 we recall definition of the class $\shf$, the fundamental class of functions and introduce the class $\shfs$. In Section 3 we prove an explicit formula for the class $\shfs$ and, as a consequence, derive that for $F \in \shfs$ and $s \notin Z(F)$, the series $\left.\sum\limits_{\rho\in Z(F)}\right. ^{\ast} \frac{1}{s-\rho}$ converges to the logarithmic derivative of the completed function $\xi_F$ at $s$. In Section 4 we show that $\tau-$Li coefficients attached to a function $F \in \shfs$ are well defined and derive the $\tau-$Li criterion in this setting. In Section 5 we derive arithmetic formulas for computation of Li coefficients. In Section 6 we conduct numerical computation of $\tau-$Li coefficients for the function $F(s)= \prod\limits_{i=1}^{K} \zeta(s-\alpha_i) \zeta(s+\alpha_i)$ where $\zeta$ denotes the Riemann zeta function and $\alpha_i$ are positive real numbers. We give an effective bound on the error term in computations and evaluate coefficients for a different range of $\tau$ in the case when $K=4$ and $\alpha_i=i$, $i=1,2,3,4$. Finally, we pose certain conjectures on the asymptotic behavior of $\tau-$Li coefficients based on these numerical calculations.


\section{Certain classes of $L-$functions}

\subsection{The Selberg class of functions}

The \emph{Selberg class} of functions $\mathcal{S}$, introduced by A. Selberg in
\cite{Selberg}, is a general class of Dirichlet series $F$ satisfying the
following properties:

\begin{itemize}
\item[(i)]  (Dirichlet series) $F$ posses a Dirichlet series representation
\begin{equation} \label{Dir series}
F(s)=\sum_{n=1}^{\infty}\frac{a_{F}(n)}{n^{s}},
\end{equation}
that converges absolutely for $\re s >1.$

\item[(ii)]  (Analytic continuation) There exists an integer $m \geq 0$ such
that $(s-1)^{m}F(s)$ is an entire function of finite order.

\item[(iii)]  (Functional equation) The function $F$ satisfies the
functional equation

$$
\xi_{F}(s)=\omega \overline{\xi_{F}(1-\bar{s})},
$$
where $\xi_F$ is the completed function defined by

$$
\xi_{F}(s)=F(s)Q^{s}_{F}\prod_{j=1}^{r}\Gamma(\lambda_{j} s+\mu_{j}),
$$

with $Q_{F}>0$, $r\geq 0$, $\lambda _{j}>0$, $\left\vert \omega \right\vert =1$, $%
\re \mu _{j}\geq 0$, $j=1,\ldots ,r.$
\item[(iv)]  (Ramanujan conjecture) For every $\epsilon >0,~a_{F}(n)\ll
n^{\epsilon }$.

\item[(v)]  (Euler product)
$$
\log F(s)=\sum_{n=1}^{\infty }\frac{b_{F}(n)}{n^{s}},
$$
where $b_{F}(n)=0$, for all $n\neq p^{m}$ with $m\geq 1$ and $p$ prime, and $%
b_{F}(n)\ll n^{\theta }$, for some $\theta <\frac{1}{2}$.
\end{itemize}

The \emph{extended Selberg class} $\mathcal{S}^{\sharp }$,
introduced in \cite{Kacz-PerelliActa} is a class of functions
satisfying conditions (i), (ii) and (iii).

It is conjectured that the Selberg class coincides with the class
of all automorphic $L-$functions. However, the Ramanujan
conjecture, boundedness of coefficients in the Dirichlet series
representation of $\log L(s,\pi )$ and the bound $\re \mu
_{j}\geq 0$ on the archimedean Langlands parameters have not yet
been verified for all automorphic $L-$functions. In order to be able to apply results unconditionally to all automorphic $L-$functions, in \cite{Sm10}, a broader class of functions, denoted by $\shf$ was defined.

\subsection{Class $\mathcal{S}^{\sharp \flat }$}

The class $\mathcal{S}^{\sharp \flat }$ is a class of functions $F\in \mathcal{S}^{\sharp }$ satisfying the following modification of the axiom (v) of the Selberg class:

\begin{itemize}
\item[(v*)] (Euler sum)  The logarithmic derivative of the function $F$ possesses a
Dirichlet series representation
$$
\frac{F^{\prime }}{F}(s)=-\sum_{n=2}^{\infty }\frac{c_{F}(n)}{n^{s}},
$$
converging absolutely for $\re s>1$.

\end{itemize}

\subsection{Class $\mathcal{S}^{\sharp \flat }(\sigma_0, \sigma_1)$}

Though very broad, the class $\shf$ still does not (unconditionally) contain some number theoretic functions possessing an Euler product representation. For example, the Rankin-Selberg $L-$functions attached to a convolution of two irreducible, unitary cuspidal representations of $GL_m(\mathbb{A}_k)$ and $GL_n(\mathbb{A}_k)$ over a number field $k$ might have poles on the line $\re s =1$, different from $s=1$, hence they do not satisfy axiom (ii). Furthermore, coefficients $\mu_j$ appearing in the functional equation (iii) for the Rankin-Selberg $L-$functions unconditionally satisfy the bound $\re \mu_j >-1$, different from the bound $\re \mu_j \geq 0$, posed in axiom (iii). In order to apply our results to all automorphic $L-$functions, as well as to zeta and $L-$products of the type $$G(s)=\prod_{i=1}^K F(s-\alpha_i)F(s+\alpha_i),$$ for some $\alpha_i \in\C$, thro\-ugh\-out this paper, we shall focus on the class $\mathcal{S}^{\sharp \flat }(\sigma_0, \sigma_1)$ of functions where axioms (i), (ii), (iii) and (v) are modified and axiom (iv) is not assumed.

Let $\sigma_0$ and $\sigma_1$ be real numbers such that $\sigma_0 \geq \sigma_1>0$. The class $\mathcal{S}^{\sharp \flat }(\sigma_0, \sigma_1)$ is the class of functions $F$ satisfying following four axioms:

\begin{enumerate}

\item[(i')] (Dirichlet series) The function $F$ possesses a Dirichlet series representation \eqref{Dir series}
that converges absolutely for $\re s >\sigma_0.$

\item[(ii')] (Analytic continuation) There exist finitely many non-negative integers $m_1, \dots, m_N$ and complex numbers $s_1, \dots, s_N$ such that the function $\prod\limits_{i=1}^{N} (s-s_i)^{m_i}F(s)$ is an entire function of finite order.
\item[(iii')]  (Functional equation)  The function $F$ satisfies the
functional equation
$$
\xi_{F}(s)=\omega \overline{\xi_{F}(\sigma_1-\bar{s})},
$$
where the completed function $\xi_F$ is defined as
\begin{eqnarray} \label{xi_F}
\xi _{F}(s)&=&F(s) Q_{F}^{s}\prod_{j=1}^{r}\Gamma (\lambda _{j}s+\mu _{j})\prod\limits_{i=1}^{2M + \delta(\sigma_1)} (s-s_i)^{m_i}\nonumber\\ &&\prod_{i=2M+1 + \delta(\sigma_1)}^{N}(s-s_i)^{m_i} (\sigma_1 - s-\overline{s_i})^{m_i},
\end{eqnarray}
where $\left| \omega \right| =1$, $Q_{F}>0$, $r\geq 0$, $\lambda _{j}>0$, $\mu_j\in\C$, $
j=1,\ldots ,r$.
Here we assume that poles of $F$ are arranged so that the first $2M+\delta(\sigma_1)$ poles ($0\leq 2M + \delta(\sigma_1) \leq N$) are such that $s_{2j-1} + \overline{s}_{2j} = \sigma_1$, for $j=1,\ldots, M$, and $\delta(\sigma_1) = 1$ if $\sigma_1/2$ is a pole of $F$ in which case $s_{2M+\delta(\sigma_1)} = \sigma_1/2$; otherwise $\delta(\sigma_1)=0$.

\item[(v')] (Euler sum)  The logarithmic derivative of the function $F$ possesses a
Dirichlet series representation
$$
\frac{F^{\prime }}{F}(s)=-\sum_{n=2}^{\infty }\frac{c_{F}(n)}{n^{s}},
$$
converging absolutely for $\re s>\sigma_0$.
\end{enumerate}

We define the \textit{non-trivial zeros} of $F$ to be the zeros of the completed function $\xi_F$. The set of non-trivial zeros of $F(s)$ is denoted by $Z(F)$. By the functional equation and axiom (v'), all those
zeros lie in the critical strip $\sigma_1-\sigma_0\leq \re s\leq \sigma_0$.
The other zeros of the function $F$ are the \textit{trivial zeros} and they arise from the poles of the gamma factors of the functional equation axiom (iii').

\begin{remark}\rm \label{rem: triv. zeros}
If a function $F\in\shfs$ has a pole $w$ such that $\sigma_1 - \overline{w}$ is not a pole of $F$, then $w=s_i$ for some $i\in\{2M+1+\delta(\sigma_1),\ldots, N\}$. In this case, the functional equation, written as $F(s)\Psi_F(s)= \overline{F}(\sigma_1-s)$, where we put $\overline{F}(s)=\overline{F(\overline{s})}$ and
\begin{equation} \label{Psi_F}
\Psi_F(s)=Q_F^{2s-\sigma_1}\omega^{-1}\prod_{j=1}^{r}\frac{\Gamma(\lambda_j s +\mu_j)}{\Gamma(\lambda_j (\sigma_1 -s)+\overline{\mu_j})}
\end{equation}
is the factor of the functional equation, implies that $\sigma_1 - \overline{w}$ is a pole of $\Psi_F(s)$ which is not a trivial zero of $F$. In other words, trivial zeros of $F$ are poles of $\Psi_F(s)$, different from $\sigma_1 - \overline{s_i}$, for $i=2M+1+\delta(\sigma_1),\ldots, N$.
\end{remark}

\begin{remark}\rm
In the definition of the completed zeta function it was necessary to distinguish between two classes of poles of $F$ i.e. the poles $z$ and $w$ such that $z+\overline{w}=\sigma_1$ and other poles since the condition $z+\overline{w} =\sigma_1$ implies that $(s-z) (s-w) (\sigma_1 -s-\overline{z}) (\sigma_1 - s -\overline{w})=(s-z)^2 (s-w)^2$, hence if one would take the product $(s-s_i)(\sigma_1 -s-\overline{s_i})$ over all $i=1,\ldots,N$ then the poles $s_i$ of $F$ of order $m_i$ such that $s_i + \overline{s_j}=\sigma_1$ for some $j\in\{1,\ldots,N\}$ could become zeros of $\xi_F$ of the same order.
\end{remark}

\begin{example}\rm\label{ExProdzeta}
Let $$F(s)=\prod_{i=1}^{K}\zeta(s-\alpha_i)\zeta(s+\alpha_i),$$
for arbitrary complex constants $\alpha_i$, where $\zeta$ denotes the Riemann zeta function. Then, $F\in\s^{\sharp\flat} (\sigma_0, 1)$ with $\sigma_0= \max\limits_{1\leq i \leq K}\{ \abs{ \re \alpha_i } +1\} $.

Namely, since function $\zeta(s\pm\alpha_i)$ has a Dirichlet series representation converging absolutely for $\re s>1+\vert\re \alpha_i\vert$, the function $F$ has a Dirichlet series representation converging absolutely for $\re s> 1+\max\limits_{1\leq i \leq K}\{ \abs{ \re \alpha_i }\}$. Furthermore, function $F$ is a meromorphic function of order one and it possesses simple poles at $s=1 \mp\alpha_i$, $i=1,\ldots, K$, hence axioms (i') and (ii') are clearly satisfied. By the functional equation satisfied by the completed zeta function

\begin{equation}\label{RiemannXi}
\xi(s)=s(s-1)\pi^{-s/2}\Gamma\left(\frac{s}{2}\right)\zeta(s)
\end{equation}
we immediately deduce that in the case when $\alpha_i \pm \overline{\alpha_j} \neq 1$, for $i,j \in\{1,\ldots,K\}$, the function
$$
\xi_F(s):= \prod_{i=1}^{K}\xi(s-\alpha_i)\xi(s+\alpha_i)
$$
satisfies axiom (iii') with $\sigma_1=1 \leq \sigma_0$, $Q_F=\pi^{-K}$, $\omega=1$, $r=2K$, $\lambda_j=1/2$ for all $j=1,\dots, 2K$, $\mu_j=\alpha_j/2$ for $j=1,\ldots, K$ and $\mu_j=-\alpha_{j-K}/2$ for $j=K+1,\ldots, 2K$.

If $ \alpha_i \pm \overline{\alpha_j} =1$ for some $i,j \in \{1,\ldots,K\}$ it is sufficient to transform the completed zeta function as
$$
\xi(s)=2(s-1)\pi^{-s/2}\Gamma\left(\frac{s}{2}+1\right)\zeta(s)
$$
to conclude that axiom (iii') is satisfied in this case as well.

Finally, since $\frac{d}{ds}(\log \zeta(s\pm\alpha_i))$ has a Dirichlet series representation converging absolutely for $\re s>1+\vert \re \alpha_i\vert$ the axiom (v') is clearly satisfied by $F$.
\end{example}

\begin{example} \rm\label{Ex: procust Fs}
Let $F_1, F_2 \in \s^{\sharp\flat}$ and $a,b \in\R$ such that $ab <0$ and $a+b+1>0$. Then, it is easy to see that $F(s)=F_1(s-a)F_2(s-b)\in\s^{\sharp\flat}(\sigma_0,\sigma_1)$, where $\sigma_0=\max\{1+a, 1+b\}$ and $0<\sigma_1=a+b+1 \leq \sigma_0$.

In particular, when $b=-a$, $a>0$ and $F\in\shf$, the function $F(s-a)F(s+a)$ belongs to $\s^{\sharp \flat} (a+1,1)$.
\end{example}

\subsection{Fundamental class of functions}

In the sequel we will show that the class $\s^{\sharp\flat}(\sigma_0,\sigma_1)$ is a subclass of a much wider class of functions, the fundamental class of
functions, introduced by J. Jorgenson and S. Lang in \cite{JorgensonLang}.

\textit{The fundamental class of functions} is a class of triples $\left( Z,%
\widetilde{Z},\Phi \right) $ satisfying the following three conditions (see \cite{JorgensonLang}, pp. 45-46):
\begin{itemize}
\item[1.] (Meromorphy) Functions $Z$ and $\widetilde{Z}$ are meromorphic functions
of a finite order.

\item[2.] (Euler sum) There are sequences $\left\{ q\right\} $ and $\left\{
\widetilde{q}\right\} $ of real numbers greater than one depending on $Z$
and $\widetilde{Z}$ such that for every $q$ and $\widetilde{q}$ there exist
complex numbers $c(q)$ and $c(\widetilde{q})$ and $\sigma _{0}^{\prime }\geq
0$ such that for all $\re s>\sigma _{0}^{\prime }$%
$$
\log Z(s)=\sum_{q}\frac{c(q)}{q^{s}}\text{ \ \ and \ \ }\log
\widetilde{Z}(s)=\sum_{\widetilde{q}}\frac{c(\widetilde{q})}{%
\widetilde{q}^{s}}.
$$%
The series are assumed to converge uniformly and absolutely in any half
plane of the form $\re s\geq \sigma _{0}^{\prime }+\epsilon >\sigma
_{0}^{\prime }$.

\item[3.] (Functional equation) There exists a meromorphic function $\Phi $ of
finite order and $\sigma _{0}$ with $0\leq \sigma _{0}\leq \sigma
_{0}^{\prime }$ such that
$$
Z(s)\Phi (s)=\widetilde{Z}(\sigma _{0}-s)
$$%
and the factor $\Phi $ of the functional equation is of a regularized
product type.
\end{itemize}

The function $\Phi $ is of a \textit{regularized product type} \cite[Def.
6.1.]{JorgensonLang} if it can be written as
\begin{equation}
\Phi (s)=e^{P(s)}Q(s)\underset{j=1}{\overset{n}{\prod }}D_{j}(\alpha
_{j}s+\beta _{j})^{k_{j}}\text{,}  \label{RegProd}
\end{equation}%
where $Q(s)$ is a rational function, $P(s)$ is a polynomial, $k_{j}$ are
integers, $D_{j}$ are regularized products and complex numbers $\alpha _{j}$
and $\beta _{j}$ are chosen such that the zeros and poles of $D_{j}$ lie in
the union of vertical strips and sectors $\left\{ z\in \mathbb{C}:\,\,
-\frac{\pi}{2}+\epsilon <\arg (z)<\frac{\pi}{2}+\epsilon \right\}$   and $ \left\{ z\in\mathbb{C}:\,\,\frac{\pi}{2}+\epsilon <\arg (z)< \frac{3\pi}{2}-\epsilon \right\}$
for some $\epsilon >0.$

The definition of a \textit{regularized product} associated to some
sequence of complex numbers is fully described in \cite[Part I, Section 2 ]%
{Jorg-Lang RP}. Since the definition is rather long, let us note here that a
regularized product can be viewed as a generalization of a Weierstrass
product. Therefore, the (classical) gamma function is a regularized product.
A \textit{reduced order of a regularized product} $D_{j}$ is defined as a
pair of numbers $\left( M_{j},m_{j}\right)$ depending on $D_{j}$ in a way
that is fully described in \cite[pp. 18-19]{JorgensonLang}. For our purposes
it is sufficient to know that a reduced order controls the growth of $D_{j}$
in vertical strips. Namely, if $\left( M_{j},m_{j}\right) $ is a reduced
order of $D_{j}$, then, $\frac{D_{j}^{\prime }}{D_{j}}(\sigma \pm iT)$ grows
at most as $T^{M_{j}}\log ^{m_{j}}\left\vert T\right\vert $, uniformly in $\sigma$, for $\sigma$ belonging to an arbitrary segment of the real line.

In particular, the gamma function is a regularized product of reduced order $%
\left( 0,0\right) $, as proved in \cite[Example 1, p. 39]{JorgensonLang}.

The notion of a reduced order of a function that is of a regularized product
type is important in the proof of the explicit formula. Namely, this order
controls the growth of functions $\frac{Z^{\prime }}{Z}$ and $\frac{\Phi
^{\prime }}{\Phi }$ in vertical strips, hence affects the conditions posed
on the test function. A \textit{reduced order of a function }$\Phi (s)$%
\textit{\ of a regularized product type} defined by (\ref{RegProd}) is $
(M,m) $ where $M=\max\limits_{j \in\{1,...,n\}} \left\{ \deg P-1,M_{j}\right\} $, $M_j$ is reduced order of $D_j$ and $m$ is the
largest of numbers $m_{j}$ such that $M_{j}=M$.



\section{An explicit formula for the class $\s^{\sharp\flat}(\sigma_0, \sigma_1)$}

In this section we prove that the class $\s^{\sharp\flat}(\sigma_0, \sigma_1)$ is a subclass of the fundamental class of functions. Then, using results of \cite{AvdispahicSmajlovicI} with $M=0$ and the modifying the evaluation of the Weil functional similarly as in  \cite{AvdispahicSmajlovicII} we prove the explicit formula for functions in the class $\s^{\sharp\flat}(\sigma_0, \sigma_1)$.

\begin{lemma} \label{order one}
Let $F\in\s^{\sharp\flat}(\sigma_0,\sigma_1)$, for some fixed $\sigma_0\geq \sigma_1 >0$. Then
the function $\xi_F$, defined by \eqref{xi_F} is entire function of order one.
\end{lemma}
\begin{proof}
 Axioms (ii') and (iii') imply that the function $\xi_F$ is an entire function of some finite order, so it is left to be proved that the order is one. By the Stirling formula, the gamma factors appearing  (iii') are bounded by $\exp(CR\log R)$ for $\vert s\vert <R$, where $C>0$ is some positive constant. The axiom (i') implies that $\xi_F(s)$ is bounded by $\exp(CR\log R)$ for $\vert s\vert <R$ and $\re s>\sigma_0$. The functional equation and boundedness of gamma factors yields the same bound for $\vert s\vert <R$ and $\re s< \sigma_1 - \sigma_0$. The application of the Phragm\'en-Lindel\" of principle in the strip $\sigma_1 - \sigma_0 \leq \re s\leq \sigma_0$ implies that the maximum modulus of $\xi_F$ in the disc $\vert s\vert <R$ is bounded by $\exp(CR\log R)$. Therefore, the function $\xi_F$ is of order at most one. Since for real $s$, $\log \xi_F(s) \sim C_1 s\log s$, as $s \to +\infty$, for some $C_1 >0$, we conclude that $\xi_F$ is of order one.
\end{proof}

\begin{lemma}
Let $F\in \mathcal{S}^{\sharp\flat } (\sigma_0, \sigma_1) $. Then, the family of triples $\left( F,\overline{F}
,\Psi _{F}\right)$, where $\overline{F}$ and $\Psi_F$ are defined in Remark \ref{rem: triv. zeros}, belongs to the fundamental class of functions.
\end{lemma}
\begin{proof}
The first axiom of the fundamental class is satisfied, since, by Lemma \ref{order one} functions $ F$ and $\overline{F}$ are meromorphic functions of order one. The Euler sum axiom of the fundamental class is satisfied with
sequences $\left\{ q\right\} $ and $\left\{ \widetilde{q}\right\} $ taken to
be the sequence of positive integers $n\geq 2$, $c(q)=\frac{c_{F}(n)}{\log n}
$, $c(\widetilde{q})=\frac{\overline{c_{F}(n)}}{\log n}$\ and $\sigma
_{0}^{\prime }=\sigma_0$.
Finally, the functional equation axiom of the fundamental class is
satisfied with $\sigma _{0}=\sigma _{1}$. The function $\Psi
_{F}(s)$ is of a regularized product type since the gamma function
is a regularized product and numbers $\lambda _{j}$ and $\mu _{j}$
are such that the poles and the zeros of the gamma factors lie in
the union of vertical strips and sectors, described above.
Obviously, the reduced order of $\Psi _{F}\,$\ is $(0,0)$.
\end{proof}

\begin{remark}\rm
Let us note that triples $\left( F,\overline{F},\Psi _{F}\right) $%
\textit{, }where $F\in \mathcal{S}^{\#}$ need not belong to the fundamental
class since $\mathcal{S}^{\#}$ may contain functions which have zeros in all
half planes of the form Re $s>\sigma $, with $\sigma >0$. Such functions do
not have an Euler sum, since the Euler sum axiom in the fundamental class
implies non-vanishing of the function in the half plane Re $s>\sigma
_{0}^{\prime }$.
\end{remark}

\begin{lemma}
Let $F\in \mathcal{S}^{\sharp\flat } (\sigma_0, \sigma_1) $
and let \eqref{Psi_F} be the factor of the functional equation $F(s)\Psi_F(s)= \overline{F}(\sigma_1-s)$. Then the factor $\Psi_F(s)$ has no zeros or poles on the line $\re s =\sigma_1/2$.
\end{lemma}
\begin{proof}
If the factor $\Psi_F(s)$ has a pole $\sigma_1/2 +it$, $t\in\R$ then, since the gamma function has only simple poles at zero and negative integers,  we have
$
\frac{-n-\mu_j}{\lambda_j}=\frac{\sigma_1}{2}+it,
$
for some non-negative integer $n$ and some $j\in\{1,\dots,r\}$, is a simple pole of $\Gamma(\lambda_j s + \mu_j)$. Then,
$
\sigma_1 + \frac{n+\overline{\mu_j}}{\lambda_j}=\frac{\sigma_1}{2}+it
$
is a zero of $(\Gamma(\lambda_j (\sigma_1 -s)+\overline{\mu_j}))^{-1}$, hence, $\frac{\sigma_1}{2}+it$ is not a pole of $\Psi_F$. Analogously, we conclude that $\Psi_F$ does not have a zero on the line $\re s=\sigma_1/2$.
\end{proof}

Before we state the explicit formula for the triple $(F, \overline{F}, \Psi_F)$ let us introduce some notation. Assume that $F\in \shfs$ has poles of order $m_i$ at $s_i$, $i=1,\ldots,N$ and put $\eta=1+\max\limits_{1\leq i \leq N}\{ \abs{ \re s_i }\}$. Let $a$ be a positive real number such that $\sigma_0 < \sigma_1 +a$ and $[-\eta,\eta]\subseteq[-a,\sigma_1+a]$, put $c = \min\limits_j \{\re\mu_j\}$, $\lambda = \max\limits_j \{\lambda_j\}$ and define $M_{max} = \lfloor a\lambda - c \rfloor$, where $\lfloor x \rfloor$ denotes the integer part of a real number $x$. We assume that parameters $\mu_j$ are indexed such that $\re{\mu_j} \leq \re{\mu_{j+1}}$.

If $M_{max} \geq 0$ for a fixed $k \in \{0,\ldots, M_{max}\}$, we let $l_k$ and $d_k$ be nonnegative integers such that
\begin{eqnarray*}
\re\mu_j &\in& \left[-k-\lambda_j (a+\sigma_1),-\lambda_j\frac{\sigma_1}{2}-k\right), \text{ \ }\text{for}\text{ \ } j = 1,\dots, l_k, \\
\re\mu_j&=& -\lambda_j\frac{\sigma_1}{2}-k, \text{ \ \ \ \ \ \ \ \ \ \ \ \ \ \ \ \ \ \ \ \ \ \ \ \ \ \ \ }\text{for}\text{ \ } j = l_k+1,\dots, d_k, \\
\re\mu_j &\in& \left(-\lambda_j\frac{\sigma_1}{2}-k,-k+a\lambda_j\right], \text{ \ \ \ \ \ \ \ \ \ \ }\text{for}\text{ \ } j = d_k+1,\dots, r.
\end{eqnarray*}
We adopt the notation that if, for a fixed $k$ the set $\left\{j\in\{1,\ldots,r\} \mid  \re\mu_j \in \right.\\\left. \left[-k-\lambda_j (a+\sigma_1),-\lambda_j\frac{\sigma_1}{2}-k\right)\right\}$ is empty, then $l_k$ is equal to zero and the sum up to $l_k$ is empty. Analogously, if the set $\left\{j \in\{1,\ldots,r\} \mid \re\mu_j = -\lambda_j\frac{\sigma_1}{2}-k \right\}$ is empty, then $l_k =d_k$ and the sum over $j \in \{l_k +1, \ldots, d_k\}$ is empty. We adopt similar convention also if the set $\{ j \in \{1,\ldots,r\} \mid \re\mu_j \in \left(-\lambda_j\frac{\sigma_1}{2}-k,\right.\\\left.-k+a\lambda_j\right]\}$ is empty.

The following proposition is the explicit formula.

\begin{proposition} \label{explicit f-la}

Let $F \in \s^{\sharp\flat}(\sigma_0, \sigma_1)$ and let $a>0$, $M_{max}$, $l_k$ and $d_k$ be as defined above. Assume that a regularized function $G$ satisfies the
following conditions:
\begin{enumerate}
\item[(a)] $G\in BV(\mathbb{R})\cap L^{1}(\mathbb{R})$,
\item[(b)] $G(x)e^{(\sigma_1/2+a)|x|}\in BV(\mathbb{R})\cap L^{1}(%
\mathbb{R})$,
\item[(c)] $G(x)+G(-x)-2G(0)=O(|\log|x||^{-\alpha})$, as $x\rightarrow 0$,
for some $\alpha>2$.
\end{enumerate}
Then, for any non-negative integer $N_c$ such that $N_c \geq \lfloor -c \rfloor +1$ the formula
\begin{eqnarray} \label{ExplicitF}
&&\lim_{T\rightarrow \infty }\sum_{\begin{subarray}{c}\rho\in Z(F)\\ |\im\rho|\leq T
\end{subarray}}ord(\rho)\mathcal{M}_{\frac{\sigma_1}{2%
}}g(\rho)=\sum_{i=1}^N m_i \mathcal{M}_{\frac{\sigma_1}{2}}g(s_i)+2\log Q_F G(0) \\
&+& \sum_{n=2}^{\infty}\frac{-c_F(n)}{n^{\frac{\sigma_1}{2}}}g(n)+\sum_{n=2}^{\infty}\frac{\overline{-c_F(n)}%
}{n^{\frac{\sigma_1}{2}}}g\left( \frac{1}{n}\right)+ \sum_{i=2M+\delta(\sigma_1) + 1}^{N}m_i\mathcal{M}_{\frac{\sigma_1}{2}}g\left( \sigma_1 - \overline{s_i}\right) \notag\\
&+& \notag\sum_{j=1}^r\lambda_j\int_{0}^{\infty }\left( \frac{2G(0)}{t}-
\frac{e^{\left( 1-\lambda_j\frac{\sigma_1}{2}-\re\mu_j-N_c\right)t}}{1-e^{-t}}\left( G_j(-\lambda_jt)
+G_j(\lambda_jt)\right) \right)e^{-t} dt \notag \\
&-&\notag \sum_{k=0}^{M_{max}} \sum_{j=1}^{l_k} \left[\mathcal{M}_{\frac{\sigma_1}{2}}g\left(\frac{-k-\mu_j}{\lambda_j}\right)
+\mathcal{M}_{\frac{\sigma_1}{2}}g\left(\sigma_1+\frac{k+\overline{\mu_j}}{\lambda_j}\right)\right] \\
&-& \notag\sum_{k=0}^{M_{max}} \sum_{j=l_k+1}^{d_k} \mathcal{M}_{\frac{\sigma_1}{2}}g\left(\frac{\sigma_1}{2}-i\frac{\im \mu_j}{\lambda_j}\right) \\
&+&\notag \sum_{k=0}^{N_c-1}\sum_{j=1}^{l_k}\int^{\infty}_0 e^{x\frac{\sigma_1}{2}+
\frac{x}{\lambda_j}\left(\re\mu_j+k\right)}\left(G_j(x)+G_j(-x)\right) dx \\
&-& \notag\sum_{k=0}^{N_c-1}\sum_{j=d_k+1}^r \int^{\infty}_0 e^{-x\frac{\sigma_1}{2}-
\frac{x}{\lambda_j}\left(\re\mu_j+k\right)}\left(G_j(x)+G_j(-x)\right) dx
\end{eqnarray}
holds true.

Here we put $g(x)=G(-\log x)$, for $x>0$ and $G_j(x)=G(x)e^{i\frac{x}{\lambda_j}\im\mu_j}$. Furthermore, $\mathcal{M}_{\frac{\sigma_1}{2}}g (s)$ denotes the translate by $\sigma_1/2$ of the Mellin transform of $g$, evaluated at $s$. $BV(\R)$ denotes the set of functions of bounded variation. If $M_{max} <0$, then we may take $N_c=0$ and all sums  appearing in the last four lines of the formula \eqref{explicit f-la} are empty. If $c\geq 0$, we may take $N_c=0$ and the last two sums in the formula \eqref{explicit f-la} to be empty.
\end{proposition}

\begin{proof}
Let $T> \max\limits_{1\leq i \leq N}\{ \abs{ \im s_i }\}$. The proof of this proposition follows the lines of the proof of the explicit formula in \cite{AvdispahicSmajlovicI}, the only difference being the evaluation of the Weil functional.\\
The integration is done along the rectangle $R_{a,T}$ with vertices $-a-iT$, $a+\sigma_1-iT$,
$a+\sigma_1+iT$,$-a+iT$. \\
Applying the explicit formula proved in  \cite[Theorem 6.1.]{AvdispahicSmajlovicI} we get
\begin{eqnarray}
&&\sum_{-a\leq \re\rho \leq a+\sigma_1}ord(\rho)\mathcal{M}_{\frac{\sigma_1}{2%
}}g(\rho)+\sum_{-a\leq \re\kappa\leq\frac{\sigma_1}{2}} ord(\kappa%
)\mathcal{M}_{\frac{\sigma_1}{2}}g(\kappa) \label{EF} \\
&=& \sum_{n=2}^{\infty}\frac{-c_F(n)}{n^{\frac{\sigma_1}{2}}}g(n)+\sum_{n=2}^{\infty}\frac{\overline{-c_F(n)}%
}{n^{\frac{\sigma_1}{2}}}g\left( \frac{1}{n}\right)+W_{\Psi_F}(G). \notag
\end{eqnarray}
The first sum on the left hand side of the above equation is taken
over all zeros and the poles of the function $F(s)$, and the
second sum is taken over all zeros and the poles of the factor
$\Psi_F(s)$, $ord(\rho)$, respectively $ord(\kappa)$, denoting the
order of zero or minus the order of pole $\rho$
respectively $\kappa $. The last term on the right hand side is
the Weil functional. The parameter $ a $ is chosen in such a way
that all poles and non-trivial
zeros of $F$ belong to the strip  $-a \leq \re s \leq a+\sigma_1$, hence
\begin{eqnarray*}
\sum_{\begin{subarray}{c} -a\leq \re\rho\leq a+\sigma_1\\|\im\rho|\leq T
\end{subarray}}ord(\rho)\mathcal{M}_{\frac{\sigma_1}{2%
}}g(\rho) &=& \sum_{\begin{subarray}{c} \rho\in Z(F)\\|\im\rho|\leq T
\end{subarray}}ord(\rho)\mathcal{M}_{\frac{\sigma_1}{2%
}}g(\rho) \\
&+& \sum_{\widetilde{\rho}\in R_{a, T}}ord(\widetilde{\rho})\mathcal{M}_{\frac{\sigma_1}{2%
}}g(\widetilde{\rho}) - \sum_{i=1}^N m_i \mathcal{M}_{\frac{\sigma_1}{2}}g(s_i), \notag
\end{eqnarray*}
where $\widetilde{\rho}$ denotes trivial zeros of $F$. The set of trivial zeros in $R_{a,T}$ is $$\widetilde{Z}_{a,T}(F) = \left\{ \frac{-k-\mu_j}{\lambda_j} \in R_{a,T} \mid k\in \{0,\dots,M_{max} \} \right\} \setminus \{\sigma_1 - \overline{s_i} \mid i=2M+\delta(\sigma_1) +1,\ldots,N\}.$$
Let $R_{a, T} = R^-_{a,T}\cup R^+_{a,T}\cup L_T$, where $R^-_{a,T} = \left[-a, \frac{\sigma_1}{2} \right)\times [-T, T]$,
 $R^+_{a,T} = \left(\frac{\sigma_1}{2}, a + \sigma_1 \right]\times [-T, T]$ and
 $L_T=\left\{\frac{\sigma_1}{2}+it | t\in [-T, T]\right\}$. \\
Notice that $\widetilde{\rho}= \frac{-k-\mu_j}{\lambda_j}\in R^-_{a, T}$ iff $\re\mu_j\in \left(-k-\frac{\sigma_1}{2%
}\lambda_j, a\lambda_j - k \right] = L_{j, k}$  and  $\im\mu_j\in [-T\lambda_j, T\lambda_j]$; analogously,
$\widetilde{\rho}\in R^+_{a, T}$ if and only if $\re\mu_j\in \left[-k-(a+\sigma_1)\lambda_j, -k-\lambda_j \frac{\sigma_1}{2} \right) = R_{j, k}$  and  $\im\mu_j\in [-T\lambda_j, T\lambda_j]$; and also analogously $\widetilde{\rho}\in L_T$ if and only if $\re\mu_j=-k-\lambda_j \frac{\sigma_1}{2}= l_{j, k}$  and  $\im\mu_j\in [-T\lambda_j, T\lambda_j]$.


Now, we get
\begin{eqnarray}
&&\lim_{T\to\infty}\sum_{\begin{subarray}{c} -a\leq \re\rho\leq a+\sigma_1\\|\im\rho|\leq T
\end{subarray}}ord(\rho)\mathcal{M}_{\frac{\sigma_1}{2}}g(\rho) \label{NPF1}\\
&=& \sum_{\rho\in Z(F)}ord(\rho)\mathcal{M}_{\frac{\sigma_1}{2%
}}g(\rho)-\sum_{i=1}^N m_i \mathcal{M}_{\frac{\sigma_1}{2}}g(s_i)\nonumber \\
&+& \sum_{k=0}^{M_{max}} \sum_{j=1}^r \delta^-_{j, k}\mathcal{M}_{\frac{\sigma_1}{2}}g\left(\frac{-k-\mu_j}{\lambda_j}\right)
+ \sum_{k=0}^{M_{max}} \sum_{j=1}^r \delta^+_{j, k}\mathcal{M}_{\frac{\sigma_1}{2}}g\left(\frac{-k-\mu_j}{\lambda_j}\right) \notag \\
&+& \sum_{k=0}^{M_{max}} \sum_{j=1}^r \delta^l_{j, k}\mathcal{M}_{\frac{\sigma_1}{2}}g\left(\frac{\sigma_1}{2}-i\frac{\im \mu_j}{\lambda_j}\right)
-\sum_{i=2M+\delta(\sigma_1) + 1}^{N} m_i \mathcal{M}_{\frac{\sigma_1}{2}}g\left( \sigma_1 - \overline{s_i}\right),\notag
\end{eqnarray}
where $\delta^-_{j, k} = 1$ if and only if $\re\mu_j\in L_{j, k}$, $\delta^+_{j, k} = 1$ if and only if $\re\mu_j\in R_{j, k}$ and $\delta^l_{j, k} = 1$ if and only if $\re\mu_j=l_{j,k}$
and $\delta^+_{j, k} = \delta^-_{j, k}= \delta^l_{j, k}=0$, otherwise.
Let us notice, again, that if $M_{max} <0$ sums over $k \in \{0,\ldots, M_{max}\}$ are empty.

Next, we evaluate the second sum on the left hand side of  \eqref{EF}. The function $\Psi_F$ has zero $\kappa_z = \sigma_1 + \frac{k + \overline{\mu_j}}{\lambda_j}$
and pole $\kappa_p =  \frac{-k - \mu_j}{\lambda_j}$, $\left(k\in \{0, \dots, M_{max} \} \right)$; when $M_{max}<0$ there are no zeros or poles of $\Psi_F$ in the strip $-a\leq \re \kappa \leq \frac{\sigma_1}{2}$, hence all sums below are empty.

It is easy to see that $\kappa_p\in R^-_{a, T}\bigcup L_T$ if and only if $\re\mu_j\in L_{j, k}$ or $\re\mu_j=l_{j,k}$ and
$\im\mu_j\in [-T\lambda_j, T\lambda_j]$.
Also, $\kappa_z\in R^-_{a, T}\bigcup L_T$ if and only if $\re\mu_j\in R_{j, k}$ or $\re\mu_j=l_{j,k}$ and $\im\mu_j\in [-T\lambda_j, T\lambda_j]$, therefore
\begin{eqnarray}\label{NPPSI}
&&\lim_{T\to\infty}\sum_{\begin{subarray}{c} -a\leq \re\kappa\leq \frac{\sigma_1}{2}\\|\im\kappa|\leq T
\end{subarray}}ord(\kappa)\mathcal{M}_{\frac{\sigma_1}{2%
}}g(\kappa) \\
&=& \sum_{k=0}^{M_{max}} \sum_{j=1}^r \delta^+_{j, k}\mathcal{M}_{\frac{\sigma_1}{2%
}}g\left(\sigma_1+\frac{k+\overline{\mu_j}}{\lambda_j}\right)
- \sum_{k=0}^{M_{max}}\sum_{j=1}^r \delta^-_{j, k}\mathcal{M}_{\frac{\sigma_1}{2%
}}g\left(\frac{-k-\mu_j}{\lambda_j}\right) \notag
\end{eqnarray}
Using relations  (\ref{NPF1}) and (\ref{NPPSI}), we get
\begin{eqnarray}\label{LS}
&&\lim_{T\to\infty}\sum_{\begin{subarray}{c} -a\leq \re\rho\leq a+\sigma_1\\|\im\rho|\leq T
\end{subarray}}ord(\rho)\mathcal{M}_{\frac{\sigma_1}{2%
}}g(\rho) + \lim_{T\to\infty}\sum_{\kappa\in R^-_{a, T}}ord(\kappa)\mathcal{M}_{\frac{\sigma_1}{2%
}}g(\kappa) \\
&=& \sum_{\rho\in Z(F)}ord(\rho)\mathcal{M}_{\frac{\sigma_1}{2%
}}g(\rho) - \sum_{i=1}^N m_i \mathcal{M}_{\frac{\sigma_1}{2}}g(s_i)\nonumber\\
&+& \sum_{k=0}^{M_{max}} \sum_{j=1}^r \delta^+_{j, k}\mathcal{M}_{\frac{\sigma_1}{2}}g\left(\frac{-k-\mu_j}{\lambda_j}\right)+ \sum_{k=0}^{M_{max}} \sum_{j=1}^r \delta^+_{j, k}\mathcal{M}_{\frac{\sigma_1}{2%
}}g\left(\sigma_1+\frac{k+\overline{\mu_j}}{\lambda_j}\right)\nonumber\\
&+& \sum_{k=0}^{M_{max}} \sum_{j=1}^r \delta^l_{j, k}\mathcal{M}_{\frac{\sigma_1}{2}}g\left(\frac{\sigma_1}{2}-i\frac{\im \mu_j}{\lambda_j}\right)
-\sum_{i=2M+\delta(\sigma_1) + 1}^{N}m_i \mathcal{M}_{\frac{\sigma_1}{2}}g\left( \sigma_1 - \overline{s_i}\right)\notag.
\end{eqnarray}
It is left to evaluate the Weil functional. Using the functional equation for the gamma function  $\Gamma(s+1)=s \Gamma(s)$ and the definition \eqref{Psi_F} of the factor $\Psi_F$, we get that the Weil functional is given by
\begin{eqnarray}\label{Weilf}
&&W_{\Psi_F}(G) = \frac{1}{ \sqrt{2\pi} }\lim_{T\rightarrow\infty}\int\limits_{-T}^{T}\hat{G}
(t)\frac{\Psi^{\prime }_F}{\Psi_F}\left(\frac{\sigma_1}{2}+it\right) dt \\
&=& 2\log Q_F G(0)+\sum_{j=1}^r \lambda_j \frac{1}{\sqrt{2\pi}}\lim_{T\rightarrow\infty}\int\limits_{-T}^{T}\hat{G}(t)
\left[ \frac{\Gamma^{\prime }}{\Gamma}\left( \lambda_j \left( \frac{\sigma_1}{2}+it\right)
+\mu_j+N_c\right) \right. \notag \\
&+& \left. \frac{\Gamma^{\prime }}{\Gamma}\left(\lambda_j\left(\frac{\sigma_1}{2}
-it\right)+\overline{\mu_j}+N_c\right) \right] dt \notag \\
&-& \sum_{k=0}^{N_c-1}\sum_{j=1}^{l_k}
\frac{\lambda_j}{\sqrt{2\pi}} \lim_{T\rightarrow\infty}\int\limits_{-T}^{T}\frac{ \hat{G}(t)(\lambda_j \sigma_1+
2 \re\mu_j+2k)dt}{\left( \lambda_j \left(\frac{\sigma_1}{2}-it\right)
+\overline{\mu_j}+k\right)\left( \lambda_j\left(\frac{\sigma_1}{2}+it\right)+\mu_j+k\right)}  \notag \\
&-& \sum_{k=0}^{N_c-1}\sum_{j=d_k+1}^{r}
\frac{\lambda_j}{\sqrt{2\pi}} \lim_{T\rightarrow\infty}\int\limits_{-T}^{T}\frac{ \hat{G}(t)(\lambda_j \sigma_1+
2 \re\mu_j+2k)dt}{\left( \lambda_j \left(\frac{\sigma_1}{2}-it\right)
+\overline{\mu_j}+k\right)\left( \lambda_j\left(\frac{\sigma_1}{2}+it\right)+\mu_j+k\right)}. \notag
\end{eqnarray}
Let us notice that the limit of the integral appearing in the sum containing the logarithmic derivatives of the gamma functions can be written as

\begin{equation}\label{GamaWeil}
\lim_{T\rightarrow \infty }\frac{1}{\sqrt{2\pi }}\int\limits_{-T+\frac{ \im\mu_j}{ \lambda_j}%
}^{T+\frac{ \im\mu_j}{ \lambda_j}}\hat{G} \left( u-\frac{ \im\mu_j}{ \lambda_j} \right)
\left[ \frac{\Gamma ^{\prime }}{\Gamma }\left( \alpha_{N_c,j}(u) \right)
+\frac{\Gamma ^{\prime }}{\Gamma }\left(\overline{\alpha_{N_c,j}(u)} \right) \right] du,
\end{equation}
where $\alpha_{N_c,j}(u) = \lambda_j\frac{ \sigma_1}{2}+\re\mu_j+N_c+i\lambda_ju.$ The application of the Stirling formula
and the bound proved in \cite[Lemma 4.1.]{AvdispahicSmajlovicI}
implies that in the above integral we may ($%
\text{mod }o(1),$ as $T\rightarrow \infty $) replace the interval $\left(-T+\frac{\im\mu_j}{\lambda_j},
T+\frac{\im\mu_j}{\lambda_j}\right)$ by the interval $\left( -T,T\right) $, to get
$$
(\ref{GamaWeil})=\lim_{T\rightarrow \infty }\frac{1}{\sqrt{2\pi }}%
\int\limits_{-T}^{T}\left( \hat{G}_j(u)+\hat{G}_j(-u)\right) \frac{%
\Gamma ^{\prime }}{\Gamma }\left( \lambda_j\frac{\sigma_1}{2}+\re\mu_j+N_c+i\lambda_ju\right) du.
$$
Application of the classical Barner-Weil formula for the Weil
functional \cite[Corollary 6.1.]{AvdispahicSmajlovic} with $f_j(x)=G_j(x)+G_j(-x)$, $a=\lambda_j\frac{\sigma_1}{2}+
\re\mu_j+N_c>0$ and $b=\frac{1}{\lambda_j}>0$ implies that
\begin{equation}\label{WeilDioI}
(\ref{GamaWeil})=\int_{0}^{\infty }\left( \frac{2G(0)}{x}-
\frac{1}{\lambda_j}\frac{e^{\left( 1-\lambda_j\frac{\sigma_1}{2}-\re\mu_j-N_c\right)\frac{x%
}{\lambda_j}}}{1-e^{-\frac{x}{\lambda_j}}}\left( G_j(-x)+G_j(x)\right) \right)e^{-\frac{x}{\lambda_j}} dx.
\end{equation}
On the other hand, applying \cite[Lemma 8.1.]{AvdispahicSmajlovicI} we get
\begin{align}\label{WeilDioII}
&\frac{1}{\sqrt{2\pi}}%
\lim_{T\rightarrow\infty}\int\limits_{-T}^{T}\hat{G}(t)\left(\frac{1}{\lambda_j\left(\frac{\sigma_1}{2}-it\right)%
+ \overline{\mu_j}+k}+\frac{1}{\lambda_j\left(\frac{\sigma_1}{2}%
+it\right)+\mu_j+k}\right) dt\\
&= \frac{1}{\lambda_j} \left\{
\begin{array}{ll}
 \int^{\infty}_0 e^{-x\frac{\sigma_1}{2}-\frac{x}{\lambda_j}\left(\re\mu_j+k\right)}\left(G_j(x)+G_j(-x)\right) dx , & \text{  if  } \frac{\sigma_1}{2}+\frac{\re\mu_j+k}{\lambda_j}>0 \nonumber\\
-\int^{\infty}_0 e^{x\frac{\sigma_1}{2}+\frac{x}{\lambda_j}\left(\re\mu_j+k\right)}\left(G_j(x)+G_j(-x)\right) dx ,  & \text{  if  } \frac{\sigma_1}{2}+\frac{\re\mu_j+k}{\lambda_j}<0%
\end{array}
\right.\nonumber
\end{align}
Substitution $x=\lambda_j t$ in (\ref{WeilDioI}) combined with  (\ref{WeilDioII}) yields
\begin{eqnarray*}\label{psi1}
&&W_{\Psi_F}(G) = 2\log Q_F G(0) \\
&+& \sum_{j=1}^r\lambda_j\int_{0}^{\infty }\left( \frac{2G(0)}{t}-
\frac{e^{\left( 1-\lambda_j\frac{\sigma_1}{2}-\re\mu_j-N_c\right)t}}{1-e^{-t}}\left( G_j(-\lambda_jt)
+G_j(\lambda_jt)\right) \right)e^{-t} dt \notag \\
&+& \sum_{k=0}^{N_c-1}\sum_{j=1}^{l_k}\int^{\infty}_0 e^{x\frac{\sigma_1}{2}+
\frac{x}{\lambda_j}\left(\re\mu_j+k\right)}\left(G_j(x)+G_j(-x)\right) dx \notag\\
&-& \sum_{k=0}^{N_c-1}\sum_{j=d_k+1}^r \int^{\infty}_0 e^{-x\frac{\sigma_1}{2}-
\frac{x}{\lambda_j}\left(\re\mu_j+k\right)}\left(G_j(x)+G_j(-x)\right) dx .\notag
\end{eqnarray*}
This, together with (\ref{EF}) and (\ref{LS})
proves (\ref {ExplicitF}).
\end{proof}

\begin{proposition} \label{log der xi as series}
Let $F\in\s^{\sharp\flat}(\sigma_0,\sigma_1)$, for some fixed $\sigma_0\geq \sigma_1 >0$ be a function such that $0\notin Z(F)$. Then the logarithmic derivative of $\xi_F$ satisfies the relation
\begin{equation} \label{xi_f log der}
\frac{\xi_F '}{\xi_F}(s)=\left. \sum_{\rho\in Z(F)} \right.^{\ast} \frac{1}{s-\rho},
\end{equation}
for all $s\in\C \setminus Z(F)$, where the zeros are counted according to their multiplicities.
\end{proposition}
\begin{proof}

We proceed analogously as in \cite{OdSm09} and apply the explicit formula \eqref{ExplicitF} to the test function
$$
G_{s}\left( x\right) =\left\{
\begin{array}{ll}
0,&x>0 \\
1/2,&x=0 \\
\exp \left( s-\sigma_1/2\right) x,&x<0%
\end{array}
\right. \text{, \ }s\in \mathbb{C}\text{, }\re s>\sigma_1+a,
$$
where $a$ is defined in Proposition \ref{explicit f-la}. For $\re \theta \leq a+\sigma_1 < \re s $, one has
\begin{eqnarray}\label{MellinForG}
\mathcal{M}_{\frac{\sigma_1}{2}}g(\theta) &=& \int_{-\infty}^{\infty}G_s(u)e^{-\left(\theta-\frac{\sigma_1}{2}\right)u} du
=\int_{-\infty}^0 e^{-\left(s-\theta \right)u} du = \frac{1}{s-\theta}.
\end{eqnarray}
Furthermore, $g(n)=G_s(-\log n)=n^{- \left(s-\frac{\sigma_1}{2}\right)}$ and $g\left(\frac{1}{n}\right)=G_s(\log n)=0$, hence
\begin{equation}\label{SumN}
\sum_{n=2}^{\infty}\frac{-c_F(n)}{n^{\frac{\sigma_1}{2}}}g(n)
+\sum_{n=2}^{\infty}\frac{\overline{-c_F(n)}}{n^{\frac{\sigma_1}{2}}}g\left(\frac{1}{n}\right)
=\sum_{n=2}^{\infty}\frac{-c_F(n)}{n^{s}}=\frac{F ^{\prime }}{F}(s),
\end{equation}
for $\re s >\sigma_1 +a >\sigma_0$. Additionally, for $t>0$ we have
$G_j(-\lambda_j t) =e^{-\lambda_j t\left(s-\frac{\sigma_1}{2}\right)-it\im \mu_j}$ and $G_j(\lambda_j t)=0$, hence
\begin{eqnarray*}\label{integral}
&&\int_{0}^{\infty }\left( \frac{2G(0)}{t}
-\frac{e^{\left( 1-\lambda_j\frac{\sigma_1}{2}-\re\mu_j-N_c\right)t}}{1-e^{-t}}\left( G_j(-\lambda_jt)
+G_j(\lambda_jt)\right) \right)e^{-t} dt \\
&=& \int_{0}^{\infty }\left( \frac{e^{-t}}{t}-
\frac{e^{-t \left( \lambda_js+\mu_j+N_c\right)}}{1-e^{-t}}\right) dt
=\frac{\Gamma^{\prime } }{\Gamma }(\lambda_js+\mu_j+N_c),
\end{eqnarray*}
by the Gauss formula for the gamma function. Furthermore, when $\frac{\sigma_1}{2} + \frac{\re\mu_j +k}{\lambda_j} <0$ one has
\begin{equation}\label{sumPoz}
\int^{\infty}_0 e^{x\frac{\sigma_1}{2}+
\frac{x}{\lambda_j}\left(\re\mu_j+k\right)}\left(G_j(x)+G_j(-x)\right) dx
= \frac{-1}{ \sigma_1-s+\frac{k+\overline{\mu_j}}{\lambda_j}}
\end{equation}
and for $\frac{\sigma_1}{2} + \frac{\re\mu_j +k}{\lambda_j} >0$ one has
\begin{equation}\label{sumNeg}
\int^{\infty}_0 e^{-x\frac{\sigma_1}{2}-
\frac{x}{\lambda_j}\left(\re\mu_j+k\right)}\left(G_j(x)+G_j(-x)\right) dx=
 \frac{1}{s+\frac{k+\mu_j}{\lambda_j}}.
\end{equation}
Since $M_{max} = \lfloor a\lambda - c\rfloor \geq \lfloor-c \rfloor$, we may take $N_c=M_{max} +1$ and put (\ref{MellinForG}), (\ref{SumN}), (\ref{integral}), (\ref{sumPoz}) and (\ref{sumNeg}) into (\ref{ExplicitF}) to get
\begin{eqnarray}\label{explicitf}
&&\lim_{T\rightarrow \infty }\sum_{\begin{subarray}{c}\rho\in Z(F)\\ |\im\rho|\leq T
\end{subarray}}\frac{ord(\rho)}{s-\rho}
= \sum_{i=1}^N \frac{m_i}{s-s_i}+\frac{F ^{\prime }}{F}(s)+\log Q_F\\&+&\sum_{j=1}^r\lambda_j \frac{\Gamma^{\prime } }{\Gamma }(\lambda_js+\mu_j+M_{max}+1)
+ \sum_{k=0}^{M_{max}}\sum_{j=1}^{l_k}\frac{-1}{ \sigma_1-s+\frac{k+\overline{\mu_j}}{\lambda_j}}\nonumber\\
&-&\sum_{k=0}^{M_{max}}\sum_{j=d_k+1}^r \frac{1}{s+\frac{k+\mu_j}{\lambda_j}}
- \sum_{k=0}^{M_{max}} \sum_{j=1}^{l_k} \left( \frac{1}{s+\frac{k+\mu_j}{\lambda_j}}+
\frac{1}{s-\left(\sigma_1+\frac{k+\overline{\mu_j}}{\lambda_j}\right)} \right) \notag \\
&-& \sum_{k=0}^{M_{max}}\sum_{j=l_k+1}^{d_k}\frac{1}{ s-\frac{\sigma_1}{2}+i\frac{\im\mu_j}{\lambda_j}} + \sum_{i=2M+1+\delta(\sigma_1)}^{N} \frac{m_i}{s-(\sigma_1 - \overline{s_i}) }. \notag
\end{eqnarray}
The completed function $\xi_F(s)$, using the functional equation for the gamma function can be written as
\begin{eqnarray*}
\xi_F(s)&=& F(s)Q_F^s\prod_{i=1}^{N}(s-s_i)^{m_i}\prod_{i=2M+1+ \delta( \sigma_1)}^N( \sigma_1-s-\overline{s_i})^{m_i}\\&&\prod_{j=1}^r \Gamma( \lambda_j s+ \mu_j+M_{max}+1)\prod_{k=0}^{M_{max}}\frac{1}{\lambda_j s+\mu_j+k}.
\end{eqnarray*}
Taking the logarithmic derivative of $\xi_F(s)$, combined with
(\ref{explicitf}) yields

$$
\lim_{T\rightarrow \infty }\sum_{\begin{subarray}{c}\rho\in Z(F)\\ |\im\rho|\leq T
\end{subarray}}\frac{ord(\rho)}{s-\rho} =
\frac{\xi_F^{\prime }}{\xi_F}(s) + \sum_{k=0}^{M_{max}}\sum_{j=1}^{r}\frac{\lambda_j}{\lambda_js+\mu_j+k}
$$
$$
-\sum_{k=0}^{M_{max}}\sum_{j=d_k +1}^{r}\frac{1}{s+\frac{k+\mu_j}{\lambda_j}}
- \sum_{k=0}^{M_{max}}\sum_{j=1}^{l_k} \frac{1}{s+\frac{\mu_j+k}{\lambda_j}}- \sum_{k=0}^{M_{max}}\sum_{j=l_k+1}^{d_k} \frac{1}{s-\frac{\sigma_1}{2}+i\frac{\im\mu_j}{\lambda_j}}. \notag
$$
Since the four sums in the above equation cancel out, we end up with the equation
$$
\frac{\xi_F^{\prime }}{\xi_F}(s)=
\lim_{T\rightarrow \infty }\sum_{\begin{subarray}{c}\rho\in Z(F)\\ |\im\rho|\leq T
\end{subarray}}\frac{ord(\rho)}{s-\rho} = -\underset{\rho
\in Z(F)}{\sum }^{\ast } \frac{1}{s-\rho}
$$
valid for $\re s > \sigma_1 +a$. The left hand side of the above equation is well defined for all $s \in \mathbb{C}\setminus Z(F)$, hence, by meromorphic continuation, we deduce the statement of the theorem.
\end{proof}

When $0\notin Z(F)$, inserting $s=0$ into equation \eqref{xi_f log der} we immediately deduce that
\begin{equation}\label{b_f as sum}
\frac{\xi_F'}{\xi_F}(0)= -\underset{\rho
\in Z(F)}{\sum }^{\ast } \frac{1}{\rho}.
\end{equation}

\section{$\tau-$Li coefficients and $\tau-$Li criterion for the class $\shfs$}

In this section we give the precise definition of $\tau-$Li coefficients attached to functions from the class $\shfs$ and prove the Li-type criterion for the zero-free regions of functions from $\shfs$.
\begin{definition}\label{tauliDef}
Let $\tau\in[\sigma_1,+\infty)$ . For an arbitrary positive integer $n$, the $n$th $\tau-$Li coefficient associated to the $F\in\shfs$ is defined as
\begin{equation}\label{tauli}
\lambda_{F}(n,\tau)=\left.\sum_{\rho\in Z(F)}\right.^{*}\left(1-\left(\frac{\rho}{\rho-\tau}\right)^n\right),
\end{equation}
where $Z(F)$ denotes the set of non-trivial zeros of $F$.
\end{definition}
The above definition is a generalization of the corresponding definition of $\tau-$Li coefficients attached to a function from the class $\shf$ from \cite{Dr12}, in the sense that it is given for a broader class of functions.

First, we prove that the coefficients $\lambda_F(n,\tau)$ are well defined.

\begin{lemma} \label{tau-li are well def.}
Let $F\in \shfs$ and let  $\tau\in[\sigma_1,+\infty)$ be an arbitrary fixed real number such that $0,\tau\notin Z(F)$. Then the following assertions are valid.
\begin{itemize}
\item[(i)] The sum \eqref{tauli} defining $\lambda_F(n,\tau)$ is $\ast-$convergent for every positive integer $n$.
\item[(ii)] We have
$$
\re(\lambda_F(n,\tau))= \sum_{\rho\in Z(F)} \re \left( 1-\left(\frac{\rho}{\rho-\tau} \right)^n\right),
$$
where the sum on the right-hand side is absolutely convergent for every positive integer $n$.
\end{itemize}
\end{lemma}
\begin{proof}
The proof follows the lines of the proof of \cite[Lemma 2.1.1.]{Dr12} which is based on a result of Bombieri and Lagarias \cite[Lemma 1]{Bombieri}. Namely, \cite[Lemma 1]{Bombieri} applied to the set $Z(F,\tau)=\left\{\frac{\rho}{\tau}:\rho\in Z(F)\right\}$, together with Lemma \ref{order one}, Proposition \ref{log der xi as series} and formula \eqref{b_f as sum} completes the proof.
\end{proof}

Now, we are able to state and prove the Li-type criterion for the zero-free regions for the functions from the class $\shfs$.  It is based on the Li criterion for arbitrary complex multiset proved in \cite{Bombieri} and its modification derived in \cite{Dr12}.

\begin{theorem}\label{thmLiCriterion}
Let $F\in \shfs$ and let $\tau\in[\sigma_1,+\infty)$ be an arbitrary fixed real number such that $0,\tau\notin Z(F)$. The following two statements are equivalent
\begin{enumerate}
  \item [(i)] $\sigma_1-\frac{\tau}{2}\leq \re \rho\leq\frac{\tau}{2}$ for every $\rho\in Z(F)$,
  \item [(ii)] $\re \lambda_F(n,\tau)\geq 0$ for every positive integer $n$.
\end{enumerate}
\end{theorem}
\begin{proof}
For the proof we apply \cite[Theorem 1]{Bombieri} or \cite[Theorem 1.6.2]{Dr12} to the multiset $R=Z(F,\tau)=\left\{\frac{\rho}{\tau}:\rho\in Z(F)\right\}$. Since $\tau\notin Z(F)$ implies $1\notin Z(F,\tau)$ in order to apply \cite[Theorem 1]{Bombieri} it is left to prove that
\begin{equation}\label{sumAssumptBL}
  \sum_{\rho\in Z(F,\tau)}\frac{1+\abs{\re\rho}}{(1+\abs\rho)^2}= \tau \sum\limits_{\rho\in Z(F)}\frac{\tau+\abs{\re\rho}}{(\tau+\abs\rho)^2} <\infty.
  \end{equation}
Zeros $\rho\in Z(F)$ are located in the critical strip $\sigma_1 -\sigma_0 \leq \re s \leq \sigma_0$, hence $\abs{\re\rho}\leq M_{\sigma_0,\sigma_1}$, where $M_{\sigma_0,\sigma_1}=\max\{\abs{\sigma_1-\sigma_0},\abs{\sigma_0}\}$. Furthermore, $\tau\geq \sigma_1>0$, hence
$$\frac{\tau+\abs{\re\rho}}{(\tau+\abs\rho)^2} \leq (\tau+M_{\sigma_0,\sigma_1})\frac{1}{\abs\rho^2},$$
 for all $\rho\in Z(F)$.
Function $F$ is entire function of order one (Lemma \ref{order one}) such that $0\notin Z(F)$, and therefore the series $\sum\limits_{\rho\in Z(F)}\frac{1}{\abs\rho^2}$ is absolutely convergent. Summation of the above inequality over all elements of $Z(F)$ shows that
$$
\sum\limits_{\rho\in Z(F)}\frac{\tau+\abs{\re\rho}}{(\tau+\abs\rho)^2} \leq (\tau+M_{\sigma_0,\sigma_1})\sum\limits_{\rho\in Z(F)}\frac{1}{\abs\rho^2}<\infty
$$
and \eqref{sumAssumptBL} is proved.

Now we may apply \cite[Theorem 1]{Bombieri} which claims that the statements $\re\rho\leq\frac{1}{2}$ for all $\rho\in Z(F,\tau)$ and $\sum\limits_{\rho\in Z(F,\tau)}\re\left(1-\left(\frac{\rho}{\rho-1}\right)^n\right)\geq 0$ for all positive integers $n$ are equivalent. Definition of the set $Z(F,\tau)$ and the assertion (ii) from Lemma \ref{tau-li are well def.} imply that the statements $\re\rho\leq\frac{\tau}{2}$ for all $\rho\in Z(F)$ and $\re \lambda_F(n,\tau)\geq 0$ for every positive integer $n$ are equivalent.

The second inequality from the statement (i) follows from the symmetry of elements in $Z(F)$ guaranteed by functional equation (iii'). Namely, $\rho \in Z(F)$ if and only if $\sigma_1-\bar{\rho}\in Z(F)$, and thus $\re\rho\leq\frac{\tau}{2}$ is equivalent to $\re{(\sigma_1-\bar\rho)}\leq\frac{\tau}{2}$, i.e. $\re\rho\geq\sigma_1-\frac{\tau}{2}$. The proof is complete.
\end{proof}

\begin{remark}\rm\label{RemTauInterval}
The assumption $\tau \geq \sigma_1$ arises naturally in the $\tau-$Li criterion, since for $\tau <\sigma_1$, the segment $[\sigma_1 - \tau/2, \tau/2]$ is an empty set, hence the statement (i) of Theorem \ref{thmLiCriterion} makes no sense.

Furthermore, since the Dirichlet series is non-vanishing in the region of its absolute convergence, for $\tau > 2\sigma_0$, the statement $\re \rho \leq \tau/2$ is satisfied by all $\rho \in Z(F)$. Hence the natural interval for values of $\tau$ in the $\tau-$Li criterion for $F\in \shfs$ is $[\sigma_1,2\sigma_0]$.

In the sequel, if not stated otherwise, we assume that $\tau \in[\sigma_1, 2\sigma_0]$.
\end{remark}

The class $\shf(1,1)$ obviously contains $\shf$, hence the $\tau-$Li criterion for the class $\shfs$ is a generalization of the $\tau-$Li criterion for the class $\shf$ derived in \cite{Dr12}. Actually, we can say more and relate $\tau-$Li criterion for the class $\shfs$ to the generalized Riemann hypothesis in the class $\shf$. Namely, the following theorem holds true.

\begin{theorem} \label{thm: tau li implies Li}
Let $F\in \shf$, $a>0$ be an arbitrary real number and let $G_a(s)= F(s-a)F(s+a)$. Then, all non-trivial zeros of $F$ are located on the line $\re s =1/2$ if and only if $\re (\lambda_{G_a}(n, 2a+1)) \geq 0$ for all positive integers $n$.
\end{theorem}
\begin{proof}
Function $G_a$ belongs to the class $\shf(a+1, 1)$, as shown in Example \ref{Ex: procust Fs}. Therefore, by Theorem \ref{thmLiCriterion} the inequality $\re (\lambda_{G_a}(n, 2a+1)) \geq 0$ for all positive integers $n$ is equivalent to the statement that for all $\rho_{G_a} \in Z(G_a)$ one has $1-\frac{2a+1}{2} \leq \re \rho_{G_a} \leq \frac{2a+1}{2}$. On the other hand, $Z(G_a)= (Z(F) +a) \cup (Z(F) -a)$, hence $\re (\lambda_{G_a}(n, 2a+1)) \geq 0$ if and only if for all $\rho_F \in Z(F)$ one has $1-\frac{2a+1}{2} \leq \re \rho_{F} -a \leq \frac{2a+1}{2}$ and $1-\frac{2a+1}{2} \leq \re \rho_{F}+a \leq \frac{2a+1}{2}$. The last two inequalities reduce to $\re \rho_F =1/2$ and the proof is completed.
\end{proof}

\section{Arithmetic formulas for $\tau-$Li coefficients}



In this section we prove two equivalent formulas for the $\tau-$Li coefficients and derive an arithmetic formula for the computation of those coefficients in terms of the coefficients $c_F(n)$ appearing in axiom (v').

Let $d_{F}(n,z_{0})$ be the power series coefficients in the
expansion of the logarithmic derivative of
$\xi_F\left(\frac{1}{1-s}\right)$ around the point $z_{0}\neq 1$ which is not a
zero of $\xi_F\left(\frac{1}{1-s}\right),$ i.e., assume that in a small neighborhood
of $z_{0}$, we have
\begin{equation}\label{defcoefdF}
\frac{d}{ds}\log
\xi_F\left(\frac{1}{1-s}\right)=\sum_{n=0}^{\infty}d_{F}(n,z_{0})(s-z_{0})^{n}.
\end{equation}
Preforming calculations analogous to those in the proof of \cite[Lemma 2.1.2]{Dr12} and having in mind the properties of the class $\shfs$, we are able to derive alternate definitions of the $\tau-$Li coefficients attached to $F\in\shfs$.

\begin{theorem}
Let $F\in \shfs$ and let $\tau\in[\sigma_1, 2\sigma_0]$ be an arbitrary fixed real number such that $0,\tau\notin Z(F)$. For every positive integer $n$ one has
\begin{equation}\label{tauLiOtherDef}
\lambda_{F}(n,\tau)= \frac{\tau}{(n-1)!}\left[\frac{d^n}{ds^n}\left(s^{n-1}\log \xi_F(s)\right)\right]_{s=\tau}=\frac{1}{\tau^n}d_{F}\left(n-1,1-\frac{1}{\tau}\right).
\end{equation}
\end{theorem}
\begin{proof}
For the first equality in \eqref{tauLiOtherDef} we may apply formula for the $n$-th derivative of the product of two analytic functions to get
\begin{eqnarray}\label{tauLidef2}
&&\frac{\tau}{(n-1)!}\left[\frac{d^n}{ds^n}\left(s^{n-1}\log \xi_F(s)\right)\right]_{s=\tau}\\
&=&\sum_{k=0}^{n-1}\binom{n}{k}\frac{\tau^{n-k}}{(n-1-k)!}\left[\frac{d^{n-k}}{ds^{n-k}}\log \xi_F(s)\right]_{s=\tau}\notag,
\end{eqnarray}
then to calculate derivatives of the logarithm of the function $\xi_F$, we may use the Hadamard product representation
\begin{equation} \label{hadamard product for xi}
\xi_F(s)=\xi_F(0) e^{b_F s} \prod_{\rho\in Z(F)} \left( 1- \frac{s}{\rho} \right) e^{s/\rho},
\end{equation}
where $b_F=\xi_F' (0) /\xi_F(0)$. Namely, \eqref{hadamard product for xi} implies that
$$
\frac{\xi'_{F}}{\xi_F}(s)=b_F+\sum_{\rho \in Z(F)}\left(\frac{1}{s-\rho}+\frac{1}{\rho}\right),
$$
where the sum is absolutely and uniformly convergent on any closed bounded region not containing non-trivial zeros of $F$. Relation \eqref{b_f as sum} and differentiation term by  term (which is justified by the uniform convergence of the above sum) imply
$$
\left[\frac{d^{n-k}}{ds^{n-k}}\log\xi_F(s)\right]_{s=\tau}=-\left.\sum_{\rho \in Z(F)}\right.^{\ast}\frac{(n-k-1)!}{(\rho-\tau)^{n-k}},
$$
where the convergence is actually absolute for $n-k\geq 2$.
After inserting the last expression in \eqref{tauLidef2} and interchanging the sums, using the binomial theorem completes the first part of the theorem.

For the second equality in \eqref{tauLiOtherDef} we will use an analogous procedure as in \cite[Lemma 2.1.2]{Dr12}. The procedure is based on the Hadamard product representation and the uniqueness of the power series expansions. We interpretate certain expressions as geometric series. Basically, starting with \eqref{hadamard product for xi} we obtain
\begin{align*}
&\frac{d}{ds}\log \xi_F\left(\frac{1}{1-s}\right) = \frac{b_F}{(1-s)^{2}}+\sum_{\rho \in Z(F)}\left(\frac{1}{(1-s)(1-\rho+s\rho)}+\frac{1}{\rho(1-s)^{2}}\right)
\\=&\frac{b_F}{(1-s)^{2}}-\sum_{\rho \in Z(F)}\left(\frac{\tau^{2}}{\rho-\tau}
\frac{1}{1-\frac{\rho\tau(s-z_{0})}{\rho-\tau}}\frac{1}{1-\tau(s-z_{0})}
-\frac{1}{\rho(1-s)^{2}}\right)\\
=&\frac{b_F}{(1-s)^{2}}+\sum_{n=0}^{\infty}\sum_{\rho \in Z(F)}\left(\tau^{n+1}\left(1-\left(\frac{\rho}{\rho-\tau}\right)^{n+1}\right)
(s-z_{0})^{n}+\frac{2^{-n-1}}{\rho(1-s)^{2}}\right),
\end{align*}
for $s$ sufficiently close to $z_0=1-1/\tau$.
Now, by the fact that $b_F$ can be represented as $\ast-$convergent sum \eqref{b_f as sum}, splitting absolutely convergent sum in the last equation into two $\ast-$convergent sums imply
$$\frac{d}{ds}\log \xi_F\left(\frac{1}{1-s}\right)= \sum_{n=0}^{\infty}\tau^{n+1}\lambda_{F}(n+1,\tau)
(s-z_{0})^{n}.$$
The uniqueness of the series expansions and \eqref{defcoefdF} complete the proof.
\end{proof}
\begin{remark} \rm
Simple calculations preformed on the expression \eqref{tauli} from the definition of $\tau-$Li coefficients, including the binomial theorem and \eqref{b_f as sum} show that the $\tau-$Li coefficients attached to $F\in \shfs$ such that $0\not\in Z(F)$ can also be represented as
$$
\lambda_F(n,\tau)=n\tau \frac{\xi_F '}{\xi_F}(0)+ n\tau^2\sum_{\rho \in Z(F)} \frac{1}{\rho(\tau-\rho)} - \sum_{k=2}^{n} \binom{n}{k}\sum_{\rho \in Z(F)}\left(\frac{\tau}{\rho-\tau} \right)^k.$$
\end{remark}
The following theorem gives an arithmetic formula for the computation of the $\tau-$Li coefficients.
\begin{theorem} \label{thm: arithmetic f-la 1} Let $F\in \shfs$ and  $\tau\in(\sigma_0,2\sigma_0]$. For every positive integer $n$ we have
\begin{eqnarray*}
\lambda_F(n,\tau) &=&\sum_{i=2M+\delta(\sigma_1)+1}^{N}
m_i\left(1-\left(\frac{\sigma_1-\overline{s_i}}{\sigma_1-\tau-\overline{s_i}}\right)^n\right)\\
&+&
\sum_{i=1}^{N} m_i\left(1-\left(\frac{s_i}{s_i-\tau}\right)^n\right)
- \tau\sum_{m=2}^{\infty}\frac{c_F(m)}{m^{\tau}} L_{n-1}^1 (\tau \log m)\\&+&\sum_{k=1}^n\binom{n}{k}\frac{(\tau\lambda_j)^k}{(k-1)!}\sum_{j=1}^r \Psi^{(k-1)}(\lambda_j \tau+\mu_j)+n\tau \log Q_F
\end{eqnarray*}
where $\Psi(s)=\frac{\Gamma'}{\Gamma}(s)$, and $L_{n}^k$ denotes the associated Laguerre polynomial.
\end{theorem}
\begin{proof} We start with the expression \eqref{tauLiOtherDef} for the $\tau-$Li coefficients given in terms of the derivatives, so
\begin{eqnarray*}
\lambda_F(n,\tau)
&=& \sum_{k=1}^n\binom{n}{k}\frac{\tau^k}{(k-1)!}\left[\frac{d^k}{ds^k}\left(
\log F(s)\right.\right.\\&+&s\log Q_F +\sum_{j=1}^r\log \Gamma(\lambda_j s+\mu_j)\\
&+&\left.\left.\sum_{i=1}^{N} m_i \log (s-s_i)
+ \sum_{i=2M + \delta(\sigma_1)+1}^N m_i\log (\sigma_1-s-\overline{s_i}) \right)\right]_{s=\tau}.
\end{eqnarray*}
Let us first evaluate the derivatives. Axiom (v') yields that
$$
\frac{d^k}{ds^k}\log F(s)
=-\frac{d^{k-1}}{ds^{k-1}}\sum_{m=2}^{\infty}\frac{c_F(m)}{m^s}=(-1)^k\sum_{m=2}^{\infty}(\log m)^{k-1}\frac{c_F(m)}{m^s},$$
for $s > \sigma_0$.

The contribution coming from the term $s\log Q_F$ is $\log Q_F$ if $k=1$ and zero otherwise. Furthermore, we have
$$
\frac{d^k}{ds^k} m_i\log(s-s_i)=-m_i\frac{(k-1)!}{(s_i-s)^k}$$
and
$$\frac{d^k}{ds^k}m_i\log(\sigma_1-s-\overline{s_i})=-m_i\frac{(k-1)!}{(\sigma_1-s-\overline{s_i})^{k}}.
$$
The term with the gamma functions is the only one remaining. We have
$$
\frac{d^k}{ds^k}\sum_{j=1}^r\log\Gamma(\lambda_j s+\mu_j)=\sum_{j=1}^r\lambda_j^k \Psi^{(k-1)}(\lambda_j s+\mu_j).
$$
Substituting now $s=\tau$, and having in mind that, by definition of the associated Laguerre polynomials, we get
\begin{eqnarray*}
&&\sum_{k=1}^n\binom{n}{k}\frac{\tau^k}{(k-1)!}(-1)^k\sum_{m=2}^{\infty} (\log m)^{k-1}\frac{c_F(m)}{m^{\tau}} \\
&=&-\tau \sum_{m=2}^{\infty}\frac{c_F(m)}{m^{\tau}}\sum_{j=0}^{n-1} \frac{1}{j!}\binom{n}{n-1-j}(\tau \log m) ^j\\
&=& -\tau \sum_{m=2}^{\infty}\frac{c_F(m)}{m^{\tau}}L_{n-1}^1(\tau \log m),
\end{eqnarray*}
plugging everything into the formula for $\lambda_F(n,\tau)$, after simplification, we obtain the desired formula.
\end{proof}

We may now derive an arithmetic formula for $\tau\in[\sigma_1,2\sigma_0]$.
\begin{theorem} \label{thm:arith formula 2} Let $F\in \shfs$ and $\tau\in[\sigma_1,2\sigma_0] \setminus \{\sigma_1 - \overline{s_i}\, : \, i=2M+\delta(\sigma_1) +1, ..., N\}$. For every positive integer $n$, we have
\begin{eqnarray*}
\lambda_F(n,\tau)&=&\sum_{\substack{i=1\\{s_i\ne \tau}}}^{N} m_i\left(1-\left(\frac{s_i}{s_i-\tau}\right)^n\right)\\&+&\sum_{i=2M+1+\delta(\sigma_1) }^N m_i\left(1-\left(\frac{\sigma_1-\overline{s_i}}{\sigma_1-\tau-\overline{s_i}}\right)^n\right) +n\tau \log Q_F
\\&+&\sum_{k=1}^n\binom{n}{k}\tau^kb_{k-1}+\sum_{k=1}^n\binom{n}{k}\frac{(\tau\lambda_j)^k}{(k-1)!}\sum_{j=1}^r \Psi^{(k-1)}(\lambda_j \tau+\mu_j)
\end{eqnarray*}
where  $\frac{F'}{F}(s)=\sum\limits_{\ell =-1}^{\infty}b_{\ell}(s-\tau)^{\ell}$ is the Laurent expansion at $s=\tau$. In particular, if $m$ is the order of the pole of the function $F$ at $s=\tau$, then $b_{-1}=-m$, and if the function $\frac{F'}{F}(s)$ does not have a pole at $s=\tau$, then $b_{-1}=0$.
\end{theorem}
\begin{proof} Let us again start with the expression for the $\tau-$Li coefficients in terms of the derivatives given in \eqref{tauLiOtherDef}. We have
\begin{eqnarray} \label{lambda arith f-la}
\lambda_F(n,\tau)& =&
 \sum_{k=1}^n\binom{n}{k}\frac{\tau^k}{(k-1)!}\left[\frac{d^k}{ds^k}\log \left(\prod\limits_{i=1}^{N} (s-s_i)^{m_i}\right.\right.\\ &&\prod\limits_{i=2M+1+\delta(\sigma_1)}^{N} (\sigma_1 - s-\overline{s_i})^{m_i}
 \left. \left. F(s) Q_{F}^{s}\prod_{j=1}^{r}\Gamma (\lambda _{j}s+\mu _{j})\right)\right]_{s=\tau}. \notag
\end{eqnarray}
Let us now concentrate for a moment on the part $(s-\tau)^{m}F(s)$, where $m$ is the order of the  pole of the function $F(s)$ at $s=\tau$, as everything else can be treated as in the proof of Theorem \ref{thm: arithmetic f-la 1}. Let
$$
\frac{F'}{F}(s)=\sum_{\ell = -1}^{\infty} b_{\ell}(s-\tau)^{\ell}
$$
be the Laurent expansion of the logarithmic derivative of $F$ at $s=\tau$. If $\tau$ is a pole of order $m$ of the function $F$, then $b_{-1}=-m$, otherwise $m=0$ and $b_{-1}=0$. Hence
$$
\frac{d}{ds}\left(\log \left((s-\tau)^{m}F(s)\right)\right)=\frac{m}{s-\tau}+\frac{F'}{F}(s)=\sum_{\ell =0}^{\infty}b_{\ell}(s-\tau)^{\ell},
$$
and thus
\begin{eqnarray*}
\left[\frac{d^k}{ds^k}\left(\log \left((s-\tau)^{m}F(s)\right)\right)\right]_{s=\tau}
=\left[\frac{d^{k-1}}{ds^{k-1}}\left(\frac{m}{s-\tau}+\frac{F'}{F}(s)\right)\right]_{s=\tau}
=(k-1)!b_{k-1}.
\end{eqnarray*}
This completes the proof.
\end{proof}
\begin{remark}\rm
In the statement of Theorem \ref{thm:arith formula 2}, for the sake of simplicity, we assumed that $\tau \neq \sigma_1 - s_i$, for real $s_i$, when $i=2M+\delta(\sigma_1) +1,\ldots, N$. Namely, if $\tau = \sigma_1 - s_i$, for $i=2M+\delta(\sigma_1) +1,\ldots, N$ is a pole of the function $\gamma_F(s) = \prod\limits_{j=1}^{r} \Gamma(\lambda_j s + \mu_j)$ (see Remark \ref{rem: triv. zeros}), then the last sum on the right-hand side of \eqref{lambda arith f-la} is not well defined. One could overcome this issue by grouping together factors $(\tau-s)^{m_i}$ and $ \gamma_F(s)$ and computing the logarithmic derivative of the product $(\tau-s)^{m_i}\gamma_F(s)$ after one applies the functional equation for the gamma function. In this case one obtains a formula similar to \eqref{lambda arith f-la}, with $m_i$ terms being different in the last sum and some additional terms arising from the application of the functional equation. We omit this case, due to a very complicated notation and the fact that those values of $\tau$ are not of high importance in the $\tau-$Li criterion.
\end{remark}
\section{Numerical computations}

Numerical computations for the $\tau$-Li coefficients can be done using definition \eqref{tauliDef}. Obviously, for the numerical evaluation truncation of the sum over zeros needs to be done, which will produce some error term.
Additionally, computations can be done using only approximate values of zeros of corresponding functions up to some error of the given size, which will produce another error term.
Our computations are conducted for the product of suitably shifted Riemann zeta functions from the Example \ref{ExProdzeta}. In the following propositions we obtain the bounds for both error terms in this special case.
\begin{proposition}\label{PropEstimate}
Let $F(s)=\prod\limits_{i=1}^{K}\zeta(s-\alpha_i)\zeta(s+\alpha_i)$, where $\alpha_i$ are positive real numbers and let $T \geq 319$. Then,
$$
\lambda_{F}(n,\tau)=\lambda^*_{F}(n,\tau, T)+E(n,\tau,T),$$
where
$$\lambda^*_{F}(n,\tau, T)=\sum_{\begin{subarray}{c}\rho\in Z(F)\\\abs{\im{\rho}}\leq T\end{subarray}}\left(1-\left(\frac{\rho}{\rho-\tau}\right)^n\right)$$
and
\begin{eqnarray*}
   \abs{E(n, \tau, T)}&<& \frac{1}{\pi} \left(KT(4\log T+8\log 2)\left(1+\frac{\tau}{T}\right)^n
-KT(4\log T+8\log 2) \right)\\&-&\frac{1}{\pi}Kn\tau (4\log T+8\log 2)
   + \frac{15}{2}Kn\tau A \frac{\log T}{T},
\end{eqnarray*}
where $A=\max\limits_{i=1,\ldots,K,\, a\in (0,1)}\{|a+\alpha_i-\tau|,|a-\alpha_i-\tau|\}$.
\end{proposition}


\begin{proof}
Definition of the function $F$ gives us possibility to write attached $\tau$-Li coefficients as a sum over zeros of the Riemann zeta function, thus the error term to be estimated can be written in the following form
\begin{eqnarray}\label{En}
   E(n, \tau, T)&=&\sum_{i=1}^{K}\sum_{\begin{subarray}{c}\rho\in Z(\zeta)\\\abs{\im{\rho}}>T\end{subarray}}
   \left(1-\left(1+\frac{\tau}{\rho+\alpha_i-\tau}\right)^n+1-\left(1+\frac{\tau}{\rho-\alpha_i-\tau}\right)^n\right)\notag\\
   &=& - \sum_{i=1}^{K}\sum_{j=1}^{n}{n\choose j}\tau^j\sum_{\begin{subarray}{c}\rho\in Z(\zeta)\\\abs{\im{\rho}}>T\end{subarray}}
   \left(\frac{1}{(\rho+\alpha_i-\tau)^j}+\frac{1}{(\rho-\alpha_i-\tau)^j}\right)
\end{eqnarray}
where $Z(\zeta)$ denotes set of zeros of the Riemann zeta function.

Let us first estimate terms for $j\geq 2$ appearing in the above sum. We will use dyadic splits, i.e. separate zeros in the regions
 $2^hT< \abs{\im \rho}\leq 2^{h+1}T$, for $h=0,1,\ldots$.
 The number of zeros of the Riemann zeta function up the the height $T$, i.e. zeros $\rho$ such that $0\leq\im\rho\leq T$, is \cite[p. 465]{KarKor}
  \begin{equation}\label{numZerosZeta}
  N(T)=\frac{T}{2\pi}\log\left(\frac{T}{2\pi}\right)-\frac{T}{2\pi}+8.9\theta \log T,
 \end{equation}
where $|\theta|<1$. We may now calculate the number of zeros in the region  $2^hT< \im \rho\leq 2^{h+1}T$
\begin{eqnarray*}
N(2^{h+1}T)-N(2^hT)&=&\frac{2^{h+1}T}{2\pi}\log\left(\frac{2^{h+1}T}{2\pi}\right)-\frac{2^{h+1}T}{2\pi}+8.9\theta_1 \log(2^{h+1}T)\\&-&\frac{2^hT}{2\pi}\log\left(\frac{2^hT}{2\pi}\right)+\frac{2^hT}{2\pi}-8.9\theta_2 \log(2^{h}T),
\end{eqnarray*}
where $\abs{\theta_1}<1$ and $\abs{\theta_2}<1$.

The contribution coming from terms $8.9\theta_1 \log(2^{h+1}T)$ and $8.9\theta_2\log(2^hT)$ is at most $8.9(\log(2^{h+1}T)+\log(2^hT))$. Furthermore,
\begin{eqnarray*}
&&\frac{2^{h+1}T}{2\pi}\log\left(\frac{2^{h+1}T}{2\pi}\right)-\frac{2^{h+1}T}{2\pi}-\frac{2^hT}{2\pi}\log\left(\frac{2^hT}{2\pi}\right)+\frac{2^hT}{2\pi} \\ &=&\frac{2^hT}{2\pi}\log(2^{h+1}T)-\frac{2^hT}{2\pi}\log\pi-\frac{2^hT}{2\pi}
\end{eqnarray*}
Clearly, we can bound the difference $N(2^{h+1}T)-N(2^hT)$ by $\frac{2^hT}{2\pi}\log(2^{h+1}T)$ if
$$
-\frac{2^hT}{2\pi}\log\pi-\frac{2^hT}{2\pi}+8.9(\log(2^{h+1}T)+\log(2^hT))<0.
$$
Using WolframAlpha, we easily derive this inequality to hold when $T\geq 319$.
Bounds determining the strip  $2^hT< \abs{\im \rho}\leq 2^{h+1}T$ give us bounds for the term in the sum. Namely, $\abs{\rho\pm\alpha_i-\tau}\geq\abs{\im\rho}$ and thus
$$\abs{\frac{1}{(\rho+\alpha_i-\tau)^j}+\frac{1}{(\rho-\alpha_i-\tau)^j}}<\frac{2}{2^{hj}T^j}.$$
This, together with the estimate of number of zeros, for $j>1$, implies
\begin{eqnarray}\label{estimatejgr1}
&&\abs{\sum_{\begin{subarray}{c}\rho\in Z(\zeta)\\\abs{\im{\rho}}>T\end{subarray}}
   \left(\frac{1}{(\rho+\alpha_i-\tau)^j}+\frac{1}{(\rho-\alpha_i-\tau)^j}\right)}
   \end{eqnarray}
$$
\leq\sum_{h=0}^{\infty}\sum_{\begin{subarray}{c}\rho\in Z(\zeta)\\2^hT<\abs{\im{\rho}}\leq 2^{h+1}T\end{subarray}}
   \abs{\frac{1}{(\rho+\alpha_i-\tau)^j}+\frac{1}{(\rho-\alpha_i-\tau)^j}}
$$
\begin{eqnarray*}
   &<&\frac{1}{2\pi}\sum_{h=0}^{\infty}2^{2-hj+h}T^{1-j}\log(2^{h+1}T)\notag\\
   &=&\frac{2}{\pi} T^{1-j}\log 2\sum_{h=0}^{\infty}\frac{h+1}{(2^{j-1})^h}+\frac{2}{\pi}T^{1-j}\log T\sum_{h=0}^{\infty} \frac{1}{(2^{j-1})^h}\notag\\
   &\leq& \frac{8}{\pi}T^{1-j}\log 2 +\frac{4}{\pi}T^{1-j}\log T.
\end{eqnarray*}
It is left to estimate the term for $j=1$, i.e. the term
$$
R_i=\sum_{\begin{subarray}{c}\rho\in Z(\zeta)\\\abs{\im{\rho}}>T\end{subarray}}
\left(\frac{1}{\rho+\alpha_i-\tau}+\frac{1}{\rho-\alpha_i-\tau}\right)
$$
for $i=1,\ldots,K$.
We will pair zeros $\rho=a+it$, $a\in\R$, $\abs{t}>T$ with $\overline{\rho}=a-it$, in the sum defining the term with $n=1$ to get
\begin{eqnarray*}
\abs{R_i} &\leq& 2 \sum_{\begin{subarray}{c}\rho\in Z(\zeta)\\\rho=\re{\rho} + it,\, t>T\end{subarray}}\left(\frac{\abs{\re{\rho}+\alpha_i-\tau}}{(\re{\rho}+\alpha_i-\tau)^2 + t^2}+\frac{\abs{\re{\rho}-\alpha_i-\tau}}{(\re{\rho}-\alpha_i-\tau)^2 + t^2}\right)\\
&\leq& 4A \sum_{\begin{subarray}{c}\rho\in Z(\zeta)\\\rho=\re{\rho} + it,\, t>T\end{subarray}}\left(\frac{1}{B^2 + t^2}\right)=4A \int\limits_{T}^{\infty} \frac{dN(t)}{t^2 + B^2}\leq 4A \int\limits_{T}^{\infty} \frac{dN(t)}{t^2},
\end{eqnarray*}
where $B=\min\limits_{i=1,\ldots,K,\, a\in (0,1)}\{|a+\alpha_i-\tau|,|a-\alpha_i-\tau|\}$.

By the same line of argumentation as above, we can bound the number of zeros $\rho$ with $0\leq \im \rho\leq T$ given by \eqref{numZerosZeta} to be at most
$\frac{5}{8}T\log T$ whenever $T \geq 14$, since then
\begin{equation}\label{N(T) bound}
\frac{T}{2\pi}\log\left(\frac{T}{2\pi}\right)-\frac{T}{2\pi}+8.9\theta\log T\leq \frac{5}{8}T\log T,
\end{equation}
and when $T<14$, $N(T)=0$. Therefore, for $T\geq 14$, using integration by parts we get
\begin{equation} \label{estimateR_i}
\abs{R_i} \leq 4A \cdot \frac{5}{4}\int\limits_{T}^{\infty} \frac{\log t dt}{t^2}= 5A\frac{\log T +1}{T}\leq \frac{15}{2}A\frac{\log T}{T},
\end{equation}
for all $i=1,\ldots,K$.
Inserting estimates \eqref{estimateR_i} and \eqref{estimatejgr1} into expression for the error term given by \eqref{En} yields
\begin{eqnarray*}\label{Enestimate}
   \abs{E(n, \tau,T)}&<&\frac{1}{\pi}\sum_{i=1}^{K}\sum_{j=2}^{n}{n\choose j}\tau^j(8 T^{1-j}\log 2+4T^{1-j}\log T)+\frac{15}{2}nK\tau A\frac{\log T}{T}\\
   &=&\frac{KT}{\pi}(4\log T+8\log 2)\left(1+\frac{\tau}{T}\right)^n- \frac{KT}{\pi}(4\log T+8\log 2)\\
   &-&\frac{1}{\pi}Kn\tau (4\log T+8\log 2)+\frac{15}{2}n K\tau A\frac{\log T}{T}
\end{eqnarray*}
and the proof is completed.
\end{proof}
\begin{proposition}\label{prop22}
Assume the zeros of the function $F\in\shfs$ are known up to an error of size $\vartheta$. Let the number of zeros of function $F$ up to the height $t$, i.e. such that $0\leq\im\rho\leq t$, denoted by $N_F(t)$ be bounded by
$$
N_F(t)\leq C_F t \log t
$$
for some constant $C_F>0$ and all $t\geq t_0\geq 1$, where $t_0$ is the smallest positive imaginary part of the zeros of function $F$. Further assume that $\tau \geq 2 \max\limits_{\rho \in Z(F)} \{\re \rho\}$ and that the zeros are symmetric with respect to the $x$-axis. Then the error $\E_{app}(n,\tau,T)$ between the computation of $\lambda_{F}^{\star}(n,\tau,T)$ using actual zeros and the computational approximations is such that
$$
\abs{\E_{app}(n,\tau,T)}\leq 4C_F\frac{\vartheta n\tau(1+\log t_0)}{t_0(1-\vartheta)}.
$$

\end{proposition}
\begin{proof} Let $\rho'$ be the approximation of $\rho$. Then $|\rho-\rho'|\leq \vartheta$. Assumption posed on $\tau$ implies $\left|\frac{\rho}{\rho-\tau}\right|\leq 1$ and $ \left|\frac{\rho'}{\rho'-\tau}\right|\leq 1$. Let $S$ to be the segment joining the points $\frac{\rho}{\rho-\tau}$ and $\frac{\rho'}{\rho'-\tau}$. Notice now that the contribution for given zero can be bounded as follows
\begin{eqnarray*}
\left|\left(\frac{\rho}{\rho-\tau}\right)^n-\left(\frac{\rho'}{\rho'-\tau}\right)^n\right|&=&\left|\int_{\frac{\rho'}{\rho'-\tau}}^{\frac{\rho}{\rho-\tau}}nx^{n-1}dx\right|\leq \left|\frac{\rho}{\rho-\tau}-\frac{\rho'}{\rho'-\tau}\right|\max\limits_{x\in S}|x|^{n-1}\\ &\leq& \frac{\vartheta n\tau}{|\rho-\tau||\rho'-\tau|}\leq \frac{\vartheta n\tau}{t^2(1-\vartheta)},\end{eqnarray*}
where $t=\im\rho$ and
$\abs{t}\geq 1$.  Let us notice that the integrand is holomorphic function,
thus the integral is independent of path of integration, and we
can integrate over line segment $S$. Assumptions of the proposition
imply that $\abs{x}^{n-1}<1$ for all $x\in S$, thus the above bound holds true.

Bounding partial sum with the sum over all zeros, using integration by parts and the bound for number of zeros  we estimate the total contribution
coming from all zeros by
\begin{eqnarray*}
\abs{\E_{app}(n,\tau,T)}&\leq&\sum_{\substack{\rho\in Z(F)\\ \rho=\re\rho+it,|t|\leq T}}\frac{\vartheta n\tau}{t^2(1-\vartheta)}\leq \sum_{\substack{\rho\in Z(F)\\\rho=\re\rho+it}}\frac{\vartheta n\tau}{t^2(1-\vartheta)}\\
&\leq& \frac{4\vartheta n\tau C_F}{1-\vartheta}\int_{t_0}^{\infty}\frac{\log t dt}{t^2}= 4C_F\frac{\vartheta n\tau }{1-\vartheta} \frac{1 + \log t_0}{t_0}.\end{eqnarray*}
The proof is completed.
\end{proof}
\begin{remark}\rm
Proposition \ref{prop22} can be proved without assumption $t_0\geq1$, but in that case some bounds and corresponding integrals are different producing more complicated error term. This assumption is satisfied by the Riemann zeta functions, and thus functions used in our examples.
\end{remark}
\begin{example} \label{prop23} \rm
Let
$$
F(s)=\prod_{i=1}^K\zeta(s+\alpha_i)\zeta(s-\alpha_i),
$$
where $\alpha_i$ are real numbers. Assume that the zeros of the Riemann zeta function $\zeta(s)$ are known up to height $T$ with an error of at most size $\vartheta$ in the imaginary part and $\tau \geq 2 \max\limits_{\rho \in Z(F)} \{\re \rho\}$.

Equation \eqref{N(T) bound} implies that $N_F(t)\leq \frac{5}{4}K t\log t$, for all $t\geq t_0$, where $14.13<t_0<14.14$ therefore we may apply Proposition \ref{prop22} with $C_F=\frac{5}{4}K$ to deduce that the difference $\E_{app}(n, \tau, T)$ between the actual value of $\lambda^{\star}(n,\tau,T)$ and the computational value is such that
$$\abs{\E_{app}(n,\tau,T)}\leq 5K \frac{\vartheta n \tau}{1-\vartheta}\frac{1+\log (14.14)}{14.13} \leq 1.30 \frac{K\vartheta n \tau}{1-\vartheta}.$$

\end{example}


The error term $\E_{app}(n,\tau, T)$ in Proposition \ref{prop22} and Example \ref{prop23} is actually independent of the value of $T$, hence, in the sequel we will denote it by $\E_{app}(n,\tau)$.

The following proposition is the error estimate for our particular example.
\begin{proposition}\label{PropEx}
Consider the function
$$F(s)=\prod_{i=1}^{4}\zeta(s-\alpha_i)\zeta(s+\alpha_i),$$
where $\alpha_i=i,$ $i=1,2,3,4$. Assume the zeros of the Riemann zeta function up to height $T\geq 319$ are known with an error of size $\vartheta_0<4\cdot 10^{-9}$ in the imaginary part, and that the first nine zeros are known with an error of size $\vartheta_1<10^{-997}$ in the imaginary part. Then the errors $\E_{app}(n,\tau)$ in the computations done with the computational zeros (vs. actual zeros) are such that
$$\abs{\E_{app}(n,1)}\leq
6.7\cdot 10^{14} \frac{n\vartheta_1}{1-\vartheta_1}+7.2\cdot 10^6\frac{n\vartheta_0}{1-\vartheta_0}$$
for $\lambda_F^{\star}(n,1,T)$, when $n$ gets values from $1$ to $500$,
$$\abs{\E_{app}(n,5)}\leq
(19.5+9\cdot10^6) \frac{n\vartheta_0}{1-\vartheta_0}+ 4.9\cdot 10^{21}\frac{n \vartheta_1}{1-\vartheta_1}$$
for $\lambda_F^{\star}(n,5,T)$, when $n$ gets values from $1$ to $200$, and
$$\abs{\E_{app}(n,10)}\leq
52 \frac{n\vartheta_0}{1-\vartheta_0}
$$
for $\lambda_F^{\star}(n,10,T)$, when $n$ gets values from $1$ to $300$.
\end{proposition}
\begin{proof}Clearly, in the case $\tau=10$, we can just use the previous example with $K=4$, since the assumption for the $\tau$ is satisfied. We obtain the bound $\abs{\E_{app}(n,10)}\leq 52 \frac{n\vartheta_0 }{1-\vartheta_0}$.

In the case $\tau=5$, we can use the method of the proof of Proposition \ref{prop22} to treat all the factors of the function that have their zeros in the half-plane $\re z\leq \frac{5}{2}$, because then $\left|\frac{\rho}{\rho-\tau}\right|<1$. These factors produce the contribution to the error term bounded by
\begin{equation}\label{e5part1}
   3.9\frac{\vartheta _0n \tau}{1-\vartheta_0}\leq19.5 \frac{n\vartheta_0}{1-\vartheta_0}.
\end{equation}
The zeros lying on the lines $\re z=\frac{7}{2}$ and $\re z=\frac{9}{2}$ require a different approach. Start as in the proof of Proposition \ref{prop22}. Let again $S$ be the segment joining the points $\frac{\rho}{\rho-\tau}$ and $\frac{\rho'}{\rho'-\tau}$. We obtain the following inequalities
$$
\left|\left(\frac{\rho}{\rho-\tau}\right)^n-\left(\frac{\rho'}{\rho'-\tau}\right)^n\right|\leq
\frac{\vartheta n\tau}{|\rho-\tau||\rho'-\tau|}\max_{x\in
S}|x|^{n-1}\leq \frac{\vartheta n\tau}{t^2(1-\vartheta)}\max_{x\in
S}|x|^{n-1},$$ where $t=\im\rho$ and $\vartheta$ can be
$\vartheta_0$ or $\vartheta_1$ depending on the zero under
consideration. Write now $x=x_0+x_1$, where $|x_1|=1$. Then
$|x|\leq 1+|x_0|$. We now estimate how large $\abs{x_0}$ is. We have
$$
\left|\frac{\rho}{\rho-\tau}\right|=\left|\frac{\re\rho+it}{\re\rho-\tau+it}\right|\leq \left|\frac{\tau-\re\rho+it}{\re\rho-\tau+it}\right|+\left|\frac{2\re\rho-\tau}{\re\rho-\tau+ti}\right|\leq 1+\frac{|2\re\rho-\tau|}{t-\vartheta}.
$$
The same bound can be obtained  when $\rho$ is
replaced with $\rho'$, having in mind that $\re\rho=\re\rho'$.
Since $S$ is line segment the maximal modulus is attained at
one of its endpoints, the above bound is valid for all $x \in S$. Now, when $\tau=5$, we can conclude that
$$
|x|\leq 1+\frac{4}{t-\vartheta}
$$
for all $x\in S$. The
function $\left(1+\frac{x}{n}\right)^n$ approaches $e^x$ from
below, when $x$ is positive. We use this to bound the expression
$\left(1+\frac{4}{t-\vartheta}\right)^{n-1}$. In particular, we
have $\left(1+\frac{16}{n-1}\right)^{n-1}\leq e^{16}<9\cdot 10^6$.
Furthermore, $\frac{4}{t-\vartheta}\leq \frac{16}{n-1}$, when
$t-\vartheta\geq \frac{4(n-1)}{16}=\frac{n-1}{4}$. Since $n\leq
200$ in our case and $\frac{n-1}{4}\leq \frac{199}{4}=49.75$ we
are able to use given bound to estimate contribution to the error
term that comes from the zeros on the lines $\re z=\frac{7}{2}$
and $\re z=\frac{9}{2}$ such that $\im \rho\geq49.75$.  The
contribution coming from each such zero is at most
$$
9\cdot 10^6\frac{\vartheta_0 n \tau}{t^2(1-\vartheta_0)},
$$
and the total contribution coming from these zeros can be obtained using
nearly the same procedure as in the proof of Proposition \ref{prop22}. To bound number of zeros under consideration we use the fact that \eqref{N(T) bound} holds true with $5/8$ replaced by $1/4$ when $t\geq49.75$. We obtain the following bounds
\begin{eqnarray}\label{e5part2}
\sum_{\substack{\rho\in Z(F),\rho=\re\rho+it\\ 49.75\leq |t|\leq T\\ \re\rho=\frac{7}{2}\textrm{ or }\re\rho=\frac{9}{2}}}
9\cdot 10^6\frac{\vartheta_0 n \tau}{t^2(1-\vartheta_0)} &\leq& 9\cdot 10^6\frac{5n\vartheta_0}{1-\vartheta_0}
\sum_{\substack{\rho\in Z(F),\rho=\re\rho+it\\ 49.75\leq |t|\\ \re \rho=\frac{7}{2}\textrm{ or }\re\rho\frac{9}{2}}}\frac{1}{t^2}\nonumber \\&\leq& 9\cdot10^6\frac{n\vartheta_0 }{1-\vartheta_0}.
\end{eqnarray}
Now we only have to consider those zeros that have an imaginary part at most $49.75$, which corresponds to the first nine zeros of the Riemann zeta function. Remember that we are now only concerned about the zeros lying on the lines $\re \rho=\frac{7}{2}$ and $\re \rho=\frac{9}{2}$.
The total contribution coming from these zeros can be estimated as follows
\begin{equation}\label{e5part3}
\sum_{\substack{\rho\in Z(F),\rho=\re\rho+it\\|t|\leq 49.75\\ \re\rho=\frac{7}{2}\textrm{ or }\re\rho=\frac{9}{2}}}\frac{n \tau \vartheta_1}{t^2(1-\vartheta_1)}\left(1+\frac{4}{|t|}\right)^{n-1} \leq 4.9\cdot 10^{21}\frac{n \vartheta_1}{1-\vartheta_1},
\end{equation}
since $t>14$ and number of such zeros is 36.
Total error $\E_{app}(n,5)$ is obtained by adding up \eqref{e5part1}, \eqref{e5part2} and \eqref{e5part3}.

We can now move to the case $\tau=1$. We have $\left|\frac{\rho}{\rho-\tau}\right|\leq 1+\frac{1}{t-\vartheta}$, where $\im\rho=t$ and hence we may proceed as above. The only difference is that now we consider expressions $\left(1+\frac{1}{t-\vartheta}\right)^{n-1}$, where $n\leq 500$. When $\frac{1}{t-\vartheta}\leq \frac{16}{n-1}$, we can bound the expression $\left(1+\frac{1}{t-\vartheta}\right)^{n-1}$ by $e^{16}<9\cdot 10^6$. Since $\frac{n-1}{16}\leq\frac{499}{16}=31.1875$ the above bound  holds for all zeros for which we use the precision $\vartheta_0$. Procedure as in the previous case implies
\begin{eqnarray*}
\sum_{\substack{\rho\in Z(F), \rho=\re\rho+it\\ 49.75\leq |t|\leq T}}9\cdot 10^6\frac{\vartheta_0 n \tau}{t^2(1-\vartheta_0)}&\leq&9\cdot 10^6\frac{n\vartheta_0 }{1-\vartheta_0}\sum_{\substack{\rho\in Z(F), \rho=\re\rho+it\\ 49.75\leq |t|}}\frac{1}{t^2} \\&\leq& 7.2\cdot 10^6 \frac{n\vartheta_0 }{1-\vartheta_0},\end{eqnarray*}
giving the bound for the error term arising from zeros with imaginary part greater than $49.75$.
Let us now estimate the contribution coming from the zeros that correspond to the first nine zeros of the Riemann zeta function. Similarly as in the case $\tau=5$, we obtain the following bound
$$
\sum_{\substack{\rho\in Z(F), \rho=\re\rho+it\\ |t|\leq 49.75}}\frac{n\tau\vartheta_1}{t^2(1-\vartheta_1)}\left(1+\frac{\tau}{|t|}\right)^{n-1}\leq 
6.7\cdot 10^{14} \frac{n\vartheta_1}{1-\vartheta_1}.
$$
Adding up two contributions estimated above, we obtain the bound for $\E_{app}(n,1)$.
\end{proof}
In the following example numerical results for the $\tau$-Li coefficients for the product of eight Riemann zeta functions are presented with the error term calculated as estimated in propositions \ref{PropEstimate} and \ref{PropEx}. Calculations are done using Wolfram Mathematica 9 and tables of zeros of the Riemann zeta functions available at \cite{Odlyzko} are used. Extensive set of data is obtained in numerical computations. In the following example we give few graphs produced using selected data and at the end we pose some conjectures suggested by the obtained data.
\begin{example}\rm Let $$F(s)=\prod_{i=1}^{4}\zeta(s-\alpha_i)\zeta(s+\alpha_i),$$
where $\alpha_i=i$, $i=1,2,3,4$. Then,
$F\in\s^{\sharp\flat} (5, 1)$. According to Remark
\ref{RemTauInterval} natural interval for parameter $\tau$ is
$[1,10]$.

We use first 2001052 zeros of the Riemman zeta function  in our calculations, implying that the truncation is done at $T=1132490.658714411$. Values for the first 9 zeros are with the error $\vartheta_1<10^{-997}$ in the imaginary part, while the error for other zeros is $\vartheta_0<4\cdot 10^{-9}$ in the imaginary part.

We have conducted numerical investigation of $\tau-$Li coefficients for $\tau\in\{1,2,3,4,5,6,7,8,9,10\}$ and obtained the extensive set of data. At figures \ref{lam1}-\ref{Fig:nlogn} we present results for $\tau\in\{1,5,10\}$.
Those values for $\tau$ are chosen in order to represent three different situations. Namely, assuming that the Riemann hypothesis holds true,
zeros of the function $F(s)$ are located at lines $\re s=a/2$,
where $a\in\{-7,-5,-3,-1,3,5,7,9\}$, thus in the case $\tau=1$, we have $\sigma_1 -\tau/2 = \tau/2 = 1/2$ and there are no zeros of the function $F(s)$ on the line $\re\rho =1/2$, (see assertion (i), Theorem
\ref{thmLiCriterion}). In the case $\tau=5$ some zeros are
contained in the strip $-3/2\leq\re\rho\leq 5/2$, but some
are not, while in the case $\tau=10$ all zeros are in the strip
$-4\leq\re\rho\leq 5$, determined by value of $\tau$. Let us
notice that the $\tau$-Li coefficients in this case are real, i.e.
$\re\lambda_{F}(n,\tau)=\lambda_{F}(n,\tau).$

Approximate values
$\lambda^*_{F}(n,\tau,T)$ as well as the total error term $|E (n,\tau,T)| + |\E_{app}(n,\tau)|$ for $T=1132490.658714411$, $\tau=1$ and $\tau=5$  are presented
in the figures \ref{lam1} and \ref{lam5}.
\begin{figure}[!htb]
 \begin{center}
  \includegraphics[width=6.5cm]{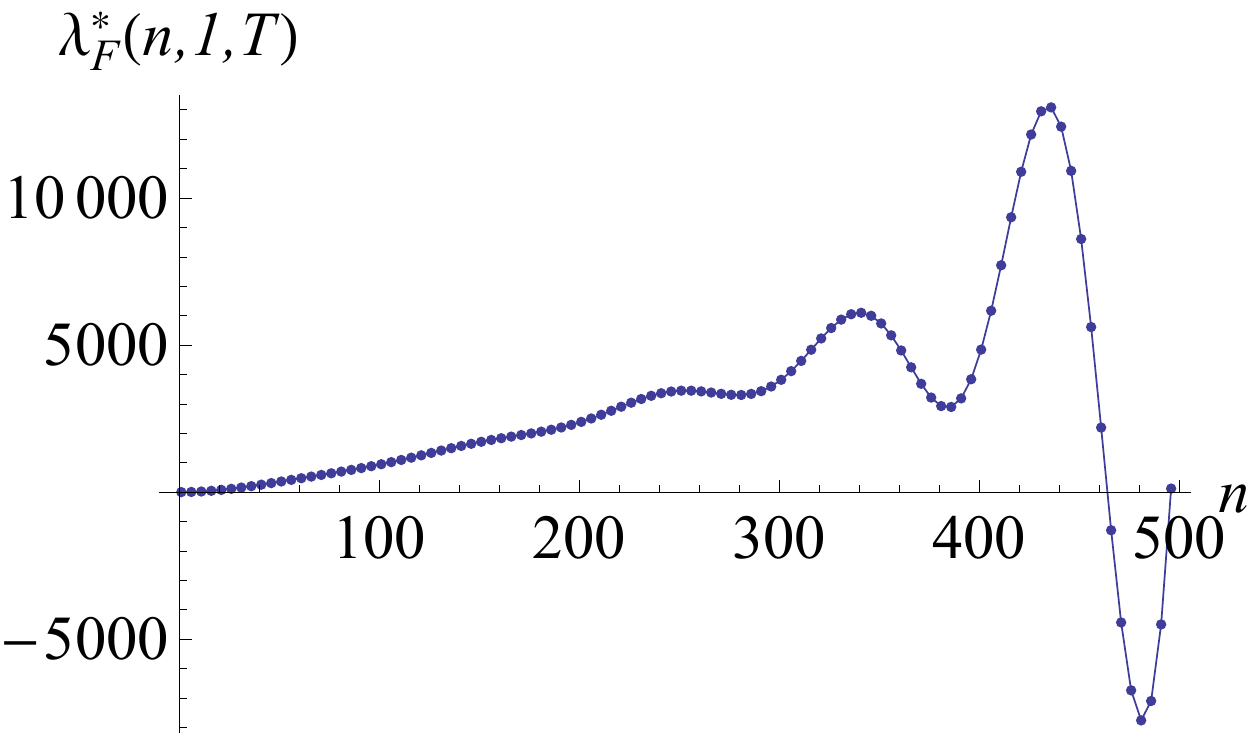} \includegraphics[width=6.5cm]{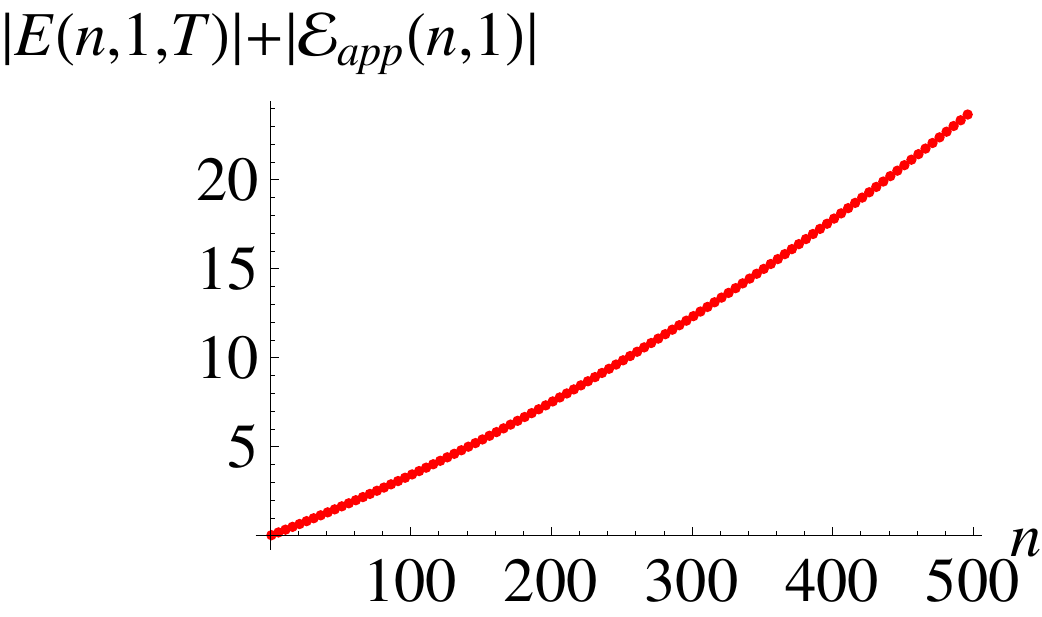}\\
  \caption{Coefficients $\lambda^*_{F}(n,1,T)$ with $T=1132490.658714411$ and bounds for the error term for $n$ from 1 to 500 calculated with step 5 }\label{lam1}
 \end{center}
\end{figure}
\begin{figure}[!htb]
 \begin{center}
  \includegraphics[width=6.5cm]{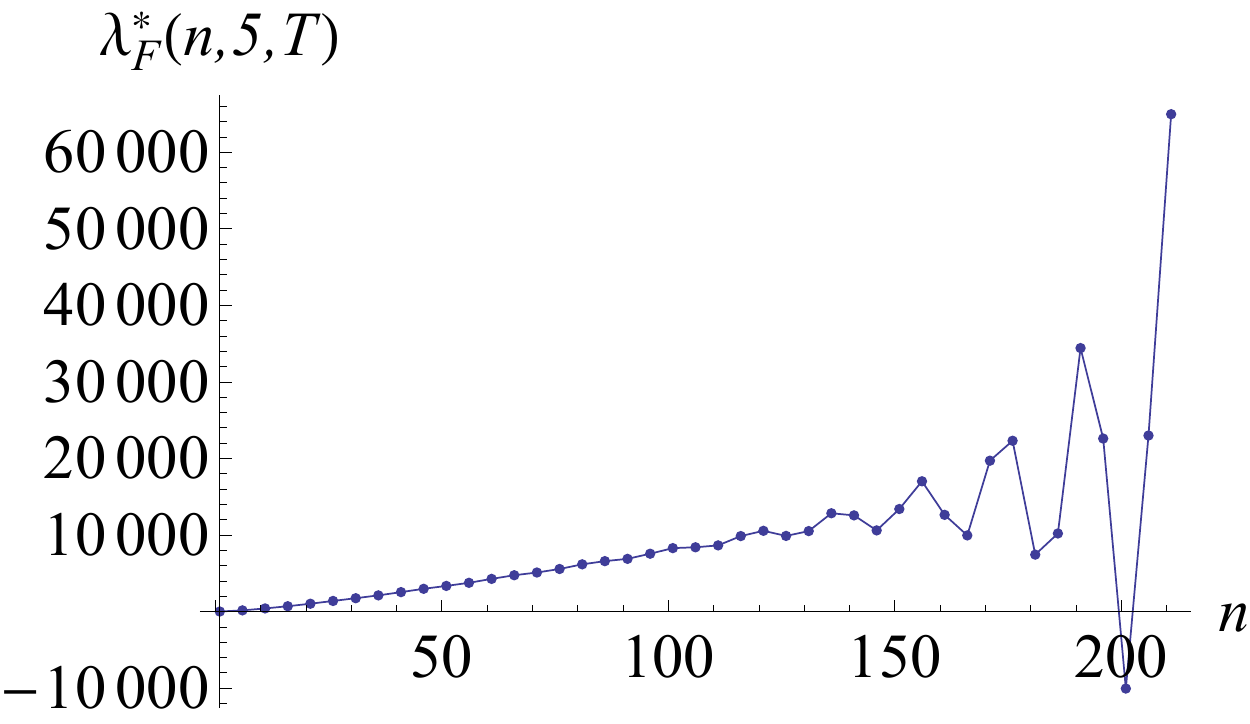} \includegraphics[width=6.5cm]{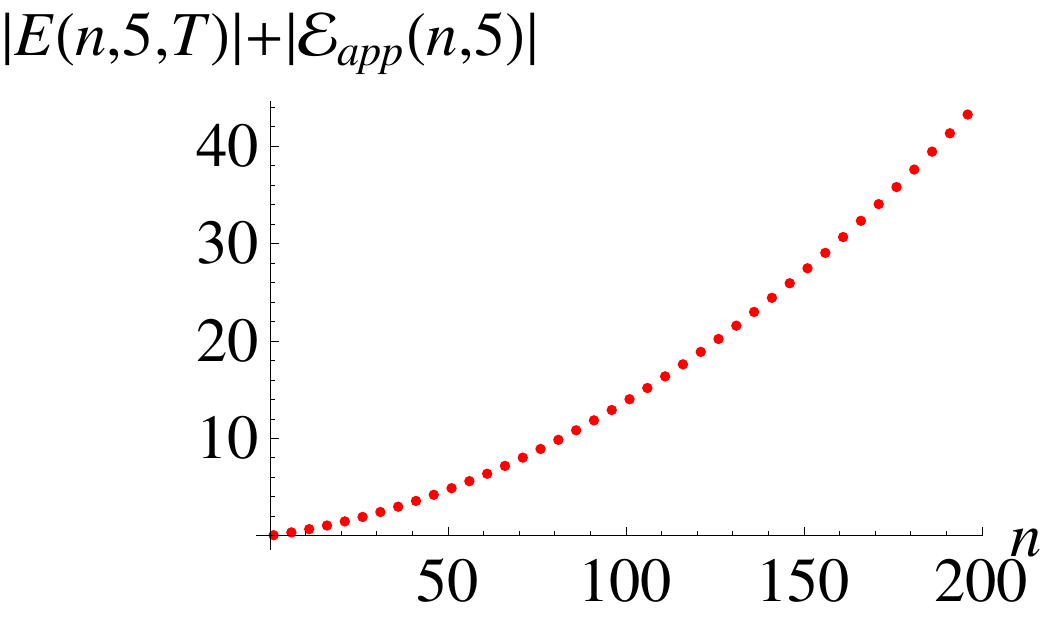}\\
  \caption{Coefficients $\lambda^*_{F}(n,5,T)$ with $T=1132490.658714411$ and bounds for the error terms for $n$ from 1 to 200 calculated with step 5 }\label{lam5}
 \end{center}
\end{figure}
Expressions for the bounds for the error terms imply that they increase with increase of $n$, as well as with increase of $\tau$. Values for these bounds presented at figures  \ref{lam1} and \ref{lam5} may seem large, but having in mind the order of magnitude of approximate values $\lambda^*_{F}(n,\tau,T)$, they are negligible, since the actual values of the $\tau$-Li coefficients belong to the segment $$\left[\lambda^*_{F}(n,\tau,T)-\abs{E(n,\tau, T)}-\abs{\E_{app}(n,\tau)},\lambda^*_{F}(n,\tau, T)+\abs{E(n,\tau, T)}+\abs{\E_{app}(n,\tau)}\right].$$

Corresponding strip in the case $\tau=10$ and the approximate values of $\tau$-Li coefficients are presented in figure \ref{l10err}. As the error term is very good, the strip is narrow, therefore one part of the strip is magnified in the figure \ref{l10err}.

\begin{center}
\begin{figure}[!htb]
 \begin{center}
  \includegraphics[width=10cm]{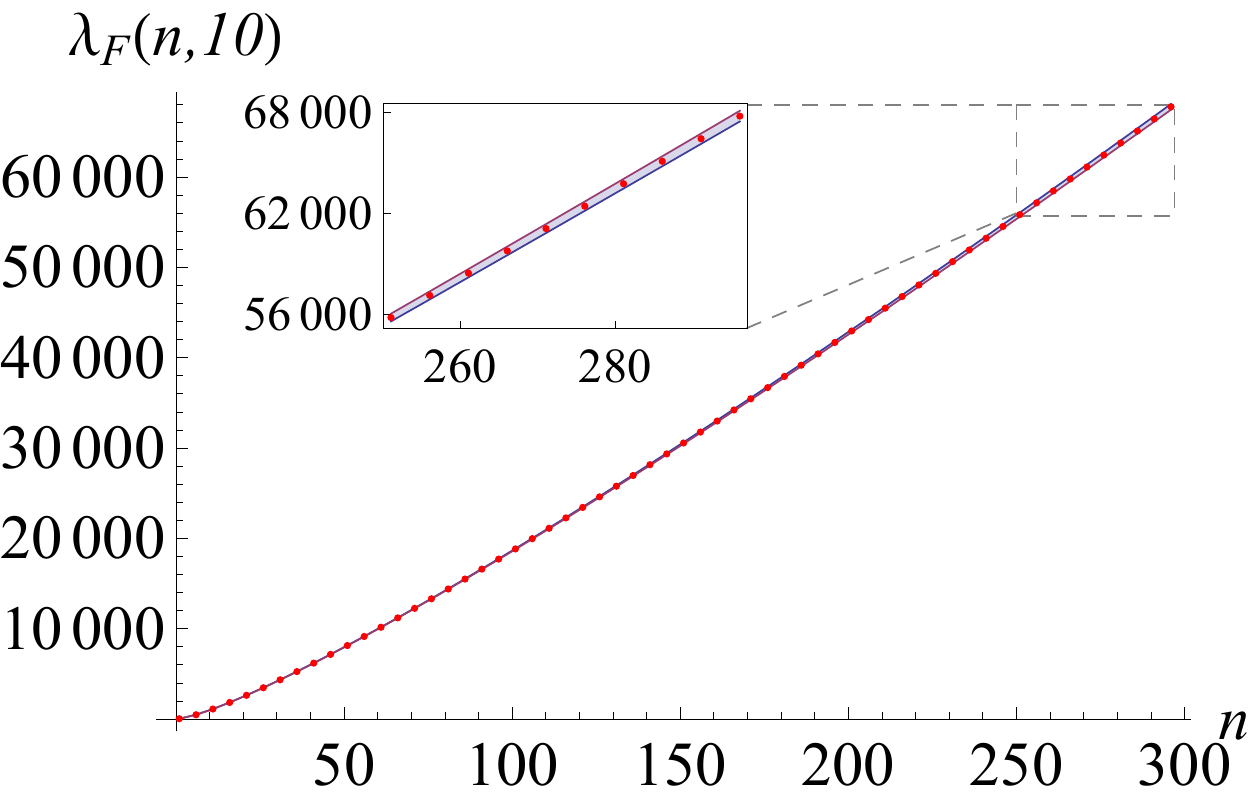}\\
  \caption{Coefficients $\lambda^*_{F}(n,10,T)$ with $T=1132490.658714411$ for $n$ for 1 to 300 calculated with step 5 (red dots) and strip (shaded area bounded by blue lines) containing actual values of coefficients $\lambda_{F}(n,10)$ determined by their approximate values and bounds for the error term }\label{l10err}
 \end{center}
\end{figure}
\end{center}
Let us emphasize that values presented in figures \ref{lam1} - \ref{l10err} are in accordance with $\tau$-Li criterion proved in Theorem \ref{thmLiCriterion}. Namely, for $\tau=1$ and $\tau=5$ assertion (i) from Theorem \ref{thmLiCriterion} obviously is not satisfied, thus some values of $\tau$-Li coefficients are negative. Such values are visible in figures \ref{lam1} and \ref{lam5}. For $\tau=10$ claim (i) from Theorem \ref{thmLiCriterion} holds true, thus $\tau$-Li coefficients are non-negative, as asserted in claim (ii), and shown in figure \ref{l10err}. Actually, the curve connecting values of $\lambda^*_{F}(n,10,T)$ is almost identical with the curve defined by the equation $y=40 x \log x$, as shown on the figure \ref{Fig:nlogn}.

\begin{center}
	\begin{figure}[!htb]
		\begin{center}
			\includegraphics[width=12cm]{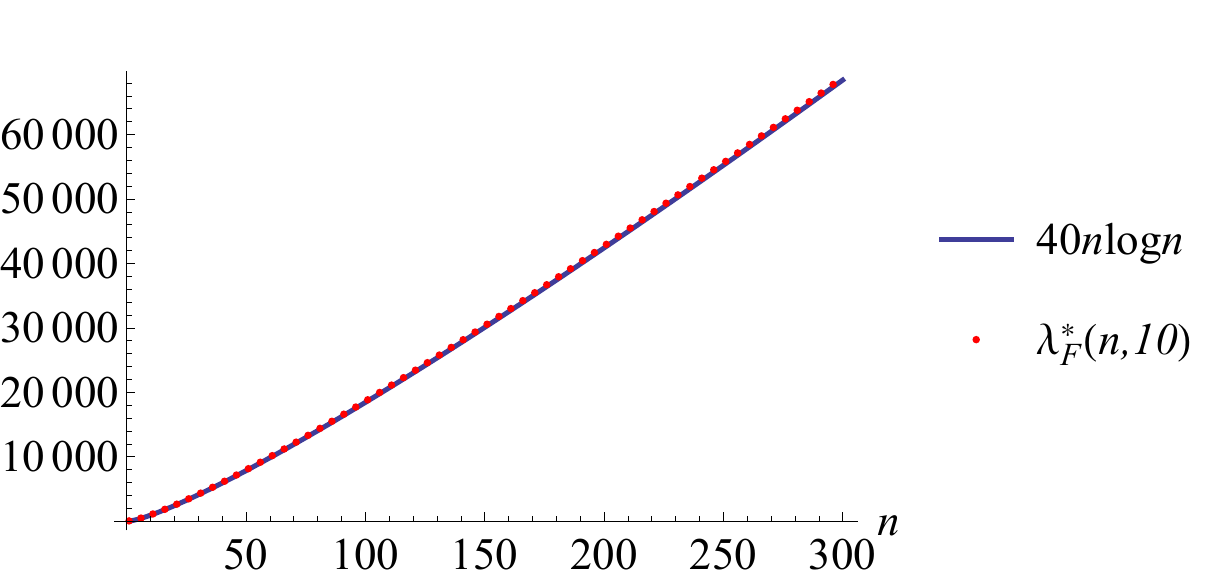}\\
			\caption{Coefficients $\lambda^*_{F}(n,10,T)$ with $T=1132490.658714411$ for $n$ for 1 to 300 calculated with step 5 and line  $y=40 n \log n$ }\label{Fig:nlogn}
		\end{center}
	\end{figure}
\end{center}

Besides that, there are some other implications suggested by the
obtained numerical evidence.
\end{example}

We conjecture that for the class $\shfs$ a
criterion relating asymptotic behavior of $\tau-$Li coefficients
and validity of $\tau-$Li criterion analogous to the criterion for the class $\shf$ and $\tau=1$ obtained in \cite{OdSm11Selberg}, Theorems 3.3. and 3.4 holds true. More precisely, we pose following two conjectures based on numerical evidence and asymptotic behavior of $\tau-$Li coefficients obtained in \cite{OdSm11Selberg} and \cite{WINEpaper}.



\begin{conjecture}
For $F\in\shfs$ non-vanishing of $F$ in the half-plane $\re(s)>\tau/2$ is equivalent to
growth of $\lambda_{F}(n,\tau)$ as $\mathcal{C}_F \tau n\log n$, as $n\to\infty$, where  $\mathcal{C}_F = \sum\limits_{j=1}^{r} \lambda_j$ and $\lambda_j$ are positive real numbers appearing in the axiom (iii').
\end{conjecture}

\begin{conjecture}
Function $F\in\shfs$ possesses non-trivial zeros in the half-plane $\re(s)>\tau/2$ if and only if coefficients $\lambda_{F}(n,\tau)$ oscillate with exponentially growing
amplitude as $n\to\infty$.
\end{conjecture}


\end{document}